\documentclass[twoside,12pt]{article}

\usepackage{amsfonts}
\usepackage{amssymb}
\usepackage{amsmath}
\usepackage{latexsym}
\usepackage{amsthm}
\usepackage{graphics}
\usepackage{graphicx}
\usepackage{multicol}
\usepackage{url}

\newtheorem{theorem}{Theorem}[section]

\newtheorem{lemma}[theorem]{Lemma}
\newtheorem{proposition}[theorem]{Proposition}
\newtheorem{definition}[theorem]{Definition}

\makeatletter
\newif\if@borderstar
\def\bordermatrix{\@ifnextchar*{%
\@borderstartrue\@bordermatrix@i}{\@borderstarfalse\@bordermatrix@i*}%
}
\def\@bordermatrix@i*{\@ifnextchar[{\@bordermatrix@ii}{\@bordermatrix@ii[()]}}
\def\@bordermatrix@ii[#1]#2{%
\begingroup
\m@th\@tempdima8.75\p@\setbox\z@\vbox{%
\def\cr{\crcr\noalign{\kern 2\p@\global\let\cr\endline }}%
\ialign {$##$\hfil\kern 2\p@\kern\@tempdima & \thinspace %
\hfil $##$\hfil && \quad\hfil $##$\hfil\crcr\omit\strut %
\hfil\crcr\noalign{\kern -\baselineskip}#2\crcr\omit %
\strut\cr}}%
\setbox\tw@\vbox{\unvcopy\z@\global\setbox\@ne\lastbox}%
\setbox\tw@\hbox{\unhbox\@ne\unskip\global\setbox\@ne\lastbox}%
\setbox\tw@\hbox{%
$\kern\wd\@ne\kern -\@tempdima\left\@firstoftwo#1%
\if@borderstar\kern2pt\else\kern -\wd\@ne\fi%
\global\setbox\@ne\vbox{\box\@ne\if@borderstar\else\kern 2\p@\fi}%
\vcenter{\if@borderstar\else\kern -\ht\@ne\fi%
\unvbox\z@\kern-\if@borderstar2\fi\baselineskip}%
\if@borderstar\kern-2\@tempdima\kern2\p@\else\,\fi\right\@secondoftwo#1 $%
}\null \;\vbox{\kern\ht\@ne\box\tw@}%
\endgroup
}
\makeatother

\title{The Tutte polynomial of some matroids}
\author{ Criel Merino\thanks{Instituto de Matem\'aticas, Universidad
      Nacional Aut\'onoma de M\'exico, Area de la Investigaci\'on
      Cient\'{\i}fica, Circuito Exterior, C.U. Coyoac\'an 04510,
      M\'exico,D.F. M\'exico. e-mail:merino@matem.unam.mx. Supported
      by Conacyt of M\'exico Proyect 83977},
       Marcelino Ram\'irez-Iba\~nez\thanks{Escuela de Ciencias, Universidad Aut\'onoma
       Benito Ju\'arez de Oaxaca, Oaxaca, M\'exico. 
       e-mail:marcelino@matem.unam.mx} \and
      Guadalupe Rodr\'iguez-S\'anchez\thanks{Departamento de Ciencias B\'asicas
      Universidad Aut\'onoma Metropolitana, Azcapozalco, Av. San Pablo No. 180,
       Col. Reynosa Tamaulipas, Azcapozalco 02200, M\'exico D.F. M\'exico.
       e-mail: rsmg@correo.azc.uam.mx}
       }  
\date{\today}
 
\begin{document}
 \maketitle
 \thispagestyle{empty} 

\begin{abstract}
 The Tutte polynomial of a graph or a matroid, named after
 W. T. Tutte, has the important universal property that essentially 
 any multiplicative graph or network invariant with a deletion and contraction
 reduction must be an evaluation of it. The deletion and contraction
 operations are natural reductions for many network models arising
 from a wide range of problems at the heart of computer science,
 engineering, optimization, physics, and biology.  Even though  the
 invariant is \#P-hard to compute in general,  there are many 
 occasions when we face the task of computing the  
 Tutte polynomial for some families of graphs or matroids. In this
 work we compile  known formulas for the Tutte polynomial of some
 families of graphs and matroids. Also, we give brief explanations of
 the techniques that were use to find the formulas. Hopefully, this will
 be useful for researchers in Combinatorics and elsewhere. 
\end{abstract}

%%%%%%%%%%%%%%%  INTRODUCTION

\section{Introduction}

  Many times as researchers in Combinatorics we  face the task of
  computing an evaluation of the Tutte polynomial of a family of graphs or matroids. 
  Sometimes  this is not an easy task, or at least  time consuming. Later, not
  surprisingly, we find out that a formula was known for a class of graphs 
  or matroids that 
  contains our family. Here we survey some of the best known formulas for
  some interesting families of graphs and matroids. Our hope is for
  researchers  to have a place to look for a Tutte polynomial
  before engaging in the search for the Tutte polynomial formula for
  the considered family. 

  We present along with the formulas, some explanation of the
  techniques used to compute them. This may also provide tools for computing  the Tutte
   polynomials of new families of graphs or matroids.  This survey can also be considered a companion
   to the book chapter Graph polynomials and their applications I: the Tutte polynomial
   by J.~Ellis-Monaghan and  C.~Merino~\cite{EM11}. There, the authors give an introduction
   of the Tutte polynomial for a general audience of scientists, pointing out relevant
  relations between
  different areas of knowledge. But very few explicit calculations are made. 
  Here we consider the 
  practical side of computing the Tutte polynomial. 

  However, we are not presenting evaluations, that is an immense area of research 
  for the Tutte polynomial, nor analysing the complexity of computing
  the invariant for the different families. For the former we strongly recommend the book of
  D.J.A. Welsh \cite{Wel93}, for the latter we recommend S.D. Noble's
  book chapter \cite{Nob07}. There are many sources for the
  theory behind this important invariant. The most useful is
  definitively J. Oxley and T. Brylawski book chapter \cite{BO92}. We already 
  mentioned~\cite{EM11},  which also surveys a variety of 
  information   about the Tutte polynomial of a graph, some of it new.
  
  We also do not address the closely related problem of characterizing families of 
  matroids by their Tutte polynomial, a problem which is generally known as Tutte 
  uniqueness. For this there are also several articles,  
  for example~\cite{BdM04,dMN04,Sar99}.

  We assume knowledge of graph theory as in \cite{Die00} and matroid
 theory as in \cite{Oxl92}.  Further details of many of the concepts
  treated here can be found in Welsh~\cite{Wel93} and Oxley and
 Brylawski~\cite{BO92}.

%%%%%%%%%%%%%%%%%%%%%% rank-nullity
%%%%%%%%%%%%%%%%%%%%%%%%%%%%%%%%%%%%%%%%%%%%%%%%%%%
\section{Definitions}
 Some of the richness of the Tutte polynomial is due to its numerous
 equivalent definitions; which is probably inherited from the vast number 
 of equivalent definitions of the concept of matroid. In this chapter 
 we revise three definitions and we put them to practice by computing
 the Tutte of some families of matroids, in particular uniform matroids.

 \subsection{The rank-nullity generating function definition}
 
  One of the simplest definitions, which is often the easiest way to 
  prove properties of the Tutte polynomial, uses the notion of rank.

 If $M=(E,r)$ is a  matroid, where $r$ is the rank-function of $M$,
 and $A\subseteq E$, we denote $r(E)-r(A)$ by $z(A)$  and $|A|-r(A)$ by
 $n(A)$, the latter is called the nullity of $A$.

\begin{definition}\label{definition}
 The Tutte polynomial of $M$, $T_{M}(x, y)$, is defined as follows:
 \begin{equation}\label{eq:expansion}
     T_{M}(x, y) = \sum_{A\subseteq E} (x-1)^{z(A)}(y-1)^{n(A)}\;.
 \end{equation}
\end{definition}

%duality
\subsubsection*{Duality}
 
 Recall that if $M=(E,r)$ is a matroid, then $M^{*}=(E,r^{*})$ is its
 dual matroid, where $r^{*}(A)=|A|-r(E)+r(E\setminus A)$. Because
 $z_{M^{*}}(A)=n_{M}(E\setminus A)$ and $n_{M^{*}}(A)=z_{M}(E\setminus
 A)$, you get the following equality
    \begin{equation}\label{eq:duality}
      T_{M}(x,y)=T_{M^{*}}(y,x).
   \end{equation}
 This gives the first technique to compute a Tutte polynomial. 
 If you know the Tutte polynomial of $M$ then you know the Tutte
   polynomial of the dual matroid $M^{*}$.

% Uniform matroids
\subsubsection*{Uniform matroids}

 Our first example is the family of uniform matroids $U_{r,n}$, where $0\leq r\leq n$. 
 Here the Tutte 
 polynomial can be computed easily using (\ref{eq:expansion}) because all subsets
  of size $k\leq r$ are 
 independent, so $n(\cdot)$ is zero;
 all subsets of size $k\geq r$ are spanning, so $z(\cdot)$ is zero and if 
 a subset is independent and 
 spanning, then it is  a basis of the matroid.
   \begin{equation}\label{eq:uniform_1}
          T_{U_{r,n}} (x,y)= \sum_{i=0}^{r-1}{n \choose i}(x-1)^{r-i}+
                           {n \choose r}+
                           \sum_{i=r+1}^{n}{n \choose i}(y-1)^{i-r}.
    \end{equation}

 Thus, for $U_{2,5}$ we get
\begin{eqnarray*}
  T_{U_{2,5}}(x,y)& = & (x-1)^2+5(x-1)+10+10(y-1)+5(y-1)^2+(y-1)^3\\
                  & = & x^2+3x+3y+2y^2+y^3.
\end{eqnarray*}
As  $U_{3,5}=(U_{2,5})^*$, we get by using equation~(\ref{eq:duality})
\begin{equation*}
  T_{U_{3,5}}(x,y)=x^3+2x^2+3x+3y+y^2.
\end{equation*}

%%%%%%%%%%%%%%%%% MATROID RELAXATION
\subsubsection*{Matroid relaxation}

Given a matroid $M=(E,r)$ with a subset $X\subseteq E$ that is both a circuit and a 
hyperplane, we can define a new matroid $M'$ as the matroid with 
basis $\mathcal{B}(M')=\mathcal{B}(M)\cup \{X\}$. That $M'$ is indeed a matroid
 is easy to check. For example, $F_{7}^{-}$ is the unique relaxation of $F_{7}$.
 The Tutte polynomial of $M'$ can be computed easily from the Tutte polynomial of
  $M$ by using equation~(\ref{eq:expansion}).

\begin{equation}\label{eq:relajacion}
 T_{M'}(x,y)=T_M(x,y)-xy+x+y.
\end{equation}

Matroid relaxation will be used extensively in the last section.

%%%%%%%%%%%% PAVING MATROIDS
\subsubsection*{Sparse paving matroids}

 Our second example can be considered a generalization of uniform matroids but it is a much larger and richer family. 
 A \emph{paving matroid} $M=(E,r)$ is a matroid whose circuits all have size
 at least $r$. Uniform matroids are an example of paving matroids. Paving matroids are closed 
 under minors and the set of 
 excluded minors for the class consists of the matroid $U_{2,2}\oplus U_{0,1}$, see for
 example~\cite{MNRV11}.  The interest about paving matroids goes back
 to 1976 when Dominic Welsh asked if  most matroids were paving, see~\cite{Oxl92}. 

  \emph{Sparse paving matroids} were introduced in~\cite{Jer06,MR08}. 
 A rank-$r$ matroid $M$ is sparse paving  if $M$ is paving and
  for every  pair of circuits $C_1$ and $C_2$ of size $r$ we have
  $|C_1\bigtriangleup C_2|>2$. For example, all uniform matroids are
  sparse paving matroids.

 There is a simple characterization of paving matroids which are sparse
 in terms of the sizes of their hyperplanes. For a proof, see~\cite{MNRV11}.
 \begin{theorem}\label{sparse_circuit_hyperplane}
   Let $M$ be a paving matroid of rank $r\geq 1$. Then $M$ is
   sparse paving if and only if all the hyperplanes of $M$ have size $r$
   or $r-1$.
 \end{theorem}

 Note that we can say a little more, any circuit of size $r$ is a hyperplane.
 Conversely, any proper subset of a hyperplane of size $r$ is independent and so such a
 hyperplane must be a circuit. Thus, the circuits of size $r$ are
 precisely the hyperplanes of size $r$.

 Many invariants that are usually difficult to compute for a general
 matroid are easy for sparse paving matroids. For example, observe
 that if $M$ is sparse paving, all subsets of size $k<r$ are independent,
 and all subsets of size $k>r$ are spanning. On the other hand the
 subsets of size $r$ are either bases or circuit--hyperplanes. Thus,
 the Tutte polynomial of a rank-$r$ sparse matroid $M$ with $n$
 elements and $\lambda$ circuit--hyperplanes is given by
 \begin{equation}\label{eq:sparse_paving}
   T_{M} (x,y)= \sum_{i=0}^{r-1}\binom ni(x-1)^{r-i}+
                           \binom nr + \lambda(xy-x-y)+
                           \sum_{i=r+1}^{n}\binom ni(y-1)^{i-r}.
  \end{equation}

 Clearly, given a $r$-rank paving matroid $M$ with $n$ elements, we can obtain the uniform matroid 
 $U_{r,n}$ by a  sequence of relaxations from $M$. If $M$ has $\lambda$ circuits-hyperplanes, 
 equation~(\ref{eq:relajacion}) also gives~(\ref{eq:sparse_paving}).

%%%%%%%%%%% free extension
\subsubsection*{Free extension}
 Another easy formula that we can obtained from the above definition involves the free extension, $M+e$, of a matroid $M=(E,r)$ by an element $e\not\in E$, which consists of adding the element  $e$ to $M$ as independently 
 as possible without increasing the rank. Equivalently, the rank function of 
 $M + e$  is given by the following 
 equations: for $X$ a subset of $E$,
 \[   r_{M+e}(X) = r_{M}(X),   \]
 and
 \begin{equation*}
   r_{M+e}(X \cup e) =
	\begin{cases}
		r_{M}(X) + 1,& \text{if $r_{M}(X) < r(M)$};\\
		r(M),&  \text{otherwise}.
	\end{cases}
 \end{equation*}
  Again, the Tutte polynomial of $M+e$ can be computed easily from the Tutte polynomial
  of $M$ by using equation~(\ref{eq:expansion}).
  \begin{equation}\label{eq:free_expansion}
           T_{M+e}(x,y)=\frac{x}{x-1}T_{M}(x,y)+\left(
                         y-\frac{x}{x-1}\right)T_{M}(1,y).
  \end{equation}
 Here the trick is to notice that $T_M(1,y)=\sum_{X} (y-1)^{|X|-r(X)}$, where
  the sum is over all subsets of  $E$ with $r(X)=r(E)$, i.e., spanning subsets
  of $M$. The presentation given here is from~\cite{BMN03}, but the formula can
  also be found in~\cite{Bry82}.
  
  For example, the graphic matroid $M(K_4\setminus e)$, has as free extention the matroid $Q_6$. Because $M(K_4\setminus e)$ is sparse paving, by using~(\ref{eq:sparse_paving}), we can compute its The Tutte polynomial and obtain
  \[
  T_{K_4\setminus e}(x,y)=x^3+2x^2+x+2xy+y+y^2.
  \]
  Now, by using~(\ref{eq:free_expansion}), we get the Tutte polynomial of the free extension.
  \[
   T_{Q_6}=x^3+3x^2+4x+2xy+4y+3y^2+y^3.
  \]
  This result can be check by using~(\ref{eq:sparse_paving}), as $Q_6$ is also sparse paving.

%%%%%%%%%%%%%%%%%%% Deletion and contraction
%%%%%%%%%%%%%%%%%%%%%%%%%%%%%%%%%%%%%%%%%%%%%%
\subsection{Deletion and contraction}

 In the second (equivalent) definition of the Tutte polynomial we use a
 linear recursion relation
 given by deleting and contracting elements that are neither loops nor
 coloops. This is by far the most used method to compute the Tutte polynomial.
 
\begin{definition}\label{def:recursive}
If $M$ is a matroid, and $e$ is an element  that is neither a loop
nor a coloop, then
\begin{equation}\label{eq:deletion-contraction}
   T_{M}(x,y) = T_{M\setminus e}(x, y) + T_{M/e}(x, y).
\end{equation}  \label{end_recursion}
 If $e$ is a coloop, then 
\begin{equation}\label{eq:coloop}
   T_{M}(x,y) = x T_{M\setminus e}(x, y).
\end{equation} 
 And if  $e$ is a loop, then 
\begin{equation}\label{eq:loop}
   T_{M}(x,y) = y T_{M/e}(x, y).
\end{equation}
\end{definition}

 The proof that Definition~\ref{definition} and~\ref{def:recursive} are
 equivalent can be found in~\cite{BO92}. 
 From this it is clear that  you just need the Tutte
 polynomial of the matroid without loops and coloops. Also, 
   if you know the Tutte polynomial of $M\setminus e$ and $M/e$, then
   you know the Tutte    polynomial of the matroid $M$. This way of computing
 the Tutte polynomial leads naturally to linear recursions. We present two examples where 
 these linear recursions give formulas. 

%Cycle C_n
\subsubsection*{The cycle $C_n$}
 As a first example of this subsection  we consider the graphic matroid $M(C_n)$. 
 Here, by deleting an edge $e$ from 
 $C_n$ we obtained a path whose Tutte polynomial is $x^{n-1}$, while if we contract $e$ we get a smaller
 cycle $C_{n-1}$. Thus, as the Tutte polynomial of the 2-cycle is $x+y$, we obtain
  \begin{equation} \label{eq:C_n}
       T_{C_n}(x,y)= \sum_{i=1}^{n-1}{x^i} + y.
  \end{equation}
  
  By duality, the Tutte polynomial of the graph $C^{*}_n$ with two vertices and $n$ parallel edges between them is 
  \begin{equation} \label{eq:C_n_dual}
       T_{C^{*}_n}(x,y)= \sum_{i=1}^{n-1}{y^i} + x.
  \end{equation}
 
 \subsubsection*{Parallel and series classes}
 As an almost trivial application of deletion-contraction we look into the common case when you have a graph
  or a matroid with parallel elements.   Let us define a \emph{parallel class} in a matroid $M$ as maximal subset $X$ of 
  $E(M)$ such that any two distinct members of $X$ are parallel and no member of $X$ is a loop , see~\cite{Oxl92}. Then
   the following result is well known and has been found many times, see~\cite{Cha95,Jac10}. 
   \begin{lemma}
    Let $X$ be a parallel class of a matroid $M$ with $|X| = p + 1$. If $X$ is not a cocircuit, then
   \begin{equation}\label{parallel_class}
     T_M(x,y) = T_{M \setminus X}(x,y) + (y^{p}+y^{p-1}+\ldots+1) T_{M/X}(x,y).   
   \end{equation}
   \end{lemma}
   The proof is by induction on $p$ with the case $p = 0$ being~(\ref{eq:deletion-contraction}). The rest of the proof 
   follows easily  from the fact that each loop introduces a multiplicative factor of $y$ in the Tutte polynomial.
   
    A \emph{series class}  in $M$ is just a parallel class in $M^{*}$ and by duality we get 
   \begin{lemma}
    Let $X$ be an series class of a matroid $M$ with $|X| = p + 1$. If $X$ is not a circuit, then
   \begin{equation}\label{series_class}
     T_M(x,y) = (x^{p}+x^{p-1}+\ldots+1)T_{M \setminus X}(x,y) +  T_{M/X}(x,y).   
   \end{equation}
   \end{lemma}

%%%%%%%%%%%%%%%%%%% square lattice
\subsubsection*{The rectangular lattice $L_{m,n}$}

 The rectangular lattice, that we will define in a moment, is our first example where 
 computing its Tutte polynomial is really far from trivial and no complete answer is 
 known. The interest resides probably in the importance of computing the
  Potts partition function, which is equivalent to the Tutte polynomial,
  of the square lattice. However, even a complete resolution of this problem will be
  just a small step towards the resolution of the really important problem of computing
  the Tutte polynomial of the cubic lattice.
  
 Let $m$ and $n$ be integers, $m, n > 1$. The \emph{grid} or \emph{rectangular lattice}
$L_{m,n}$ is a connected planar graph with $m n$ vertices, such
that its faces are squares, except one face that it is a polygon
with $2(m + n)-4$ edges. The grid graph $L_{m,n}$ can be
represented as in Figure~\ref{fig:grid_mxn}

\begin{figure}
\begin{center}
   \includegraphics[scale=0.25]{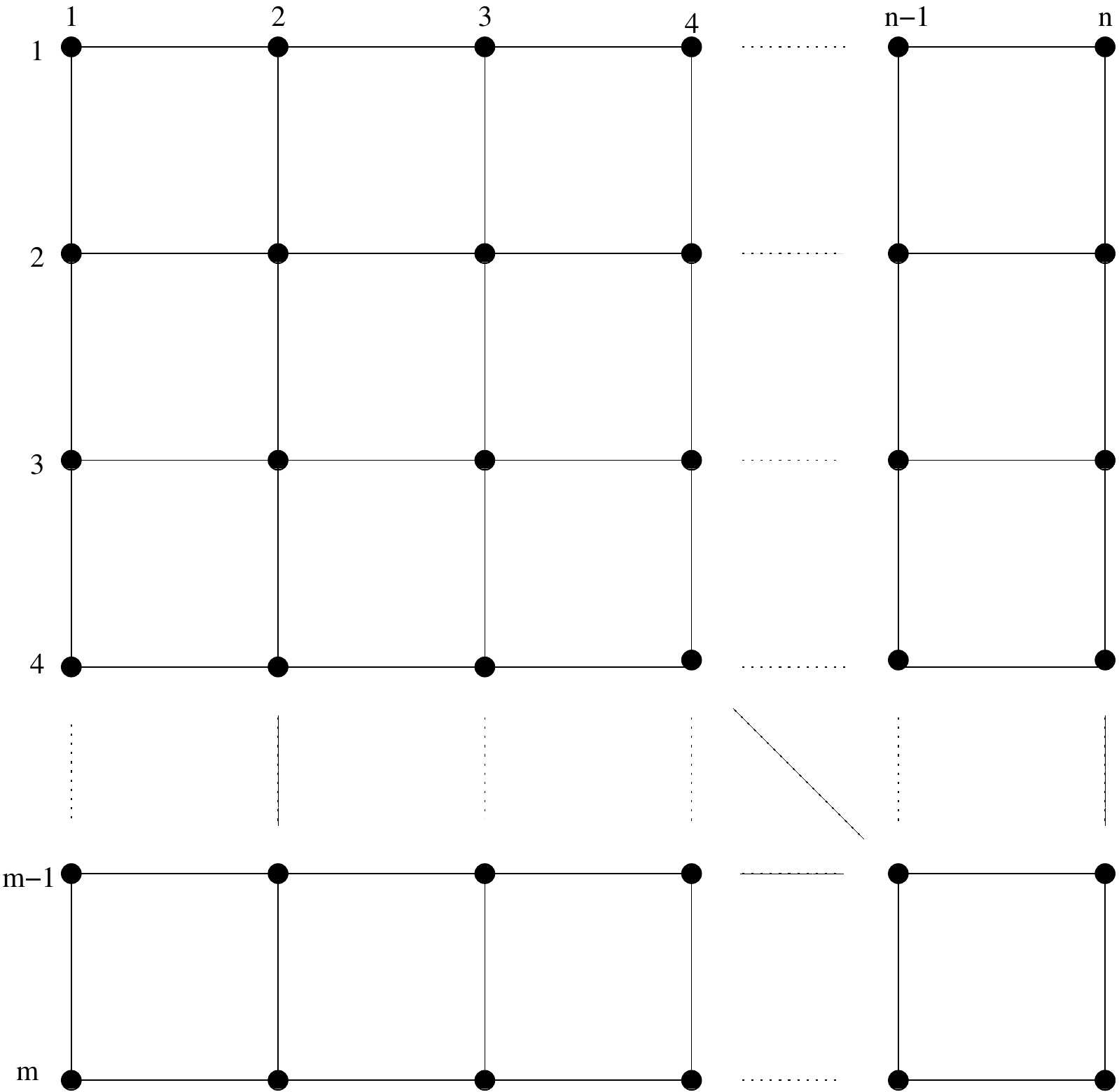}
    \caption{The grid graph $L_{m,n}$.}\label{fig:grid_mxn}
\end{center}
\end{figure}

 The vertices of $L_{m,n}$ are denoted as ordered pairs,
such that the vertex in the intersection of the row $i$ and the
column $j$ is denoted by $ij$.

 As example, we show a recursive formula for the Tutte polynomials of grid
graphs $L_{m,n}$ when   $m=2$.
For $m > 3$, the formulas are very complicated to use in 
practical calculations.

 The graphs $L_{2,1}$, $L_{2,2}$ and $L_{2,3}$ are shown
in the Figure~\ref{fig:first_terms}, with the labels that correspond to their
vertices.

For every positive integer $n$. The graph $Q_{2,n}$ is
defined from the graph $L_{2,n}$ by contraction of the edge
$(11,21)$ of $L_{2,n}$. In particular, the graph $Q_{2,1}$ is one
isolated vertex, see Figure~\ref{fig:first_terms}
%%%%%%%%%%%%%%%%
\begin{figure}[h!] 
\begin{center} 
\begin{minipage}[b]{0.4\linewidth}
\centering
   \includegraphics[scale=0.25]{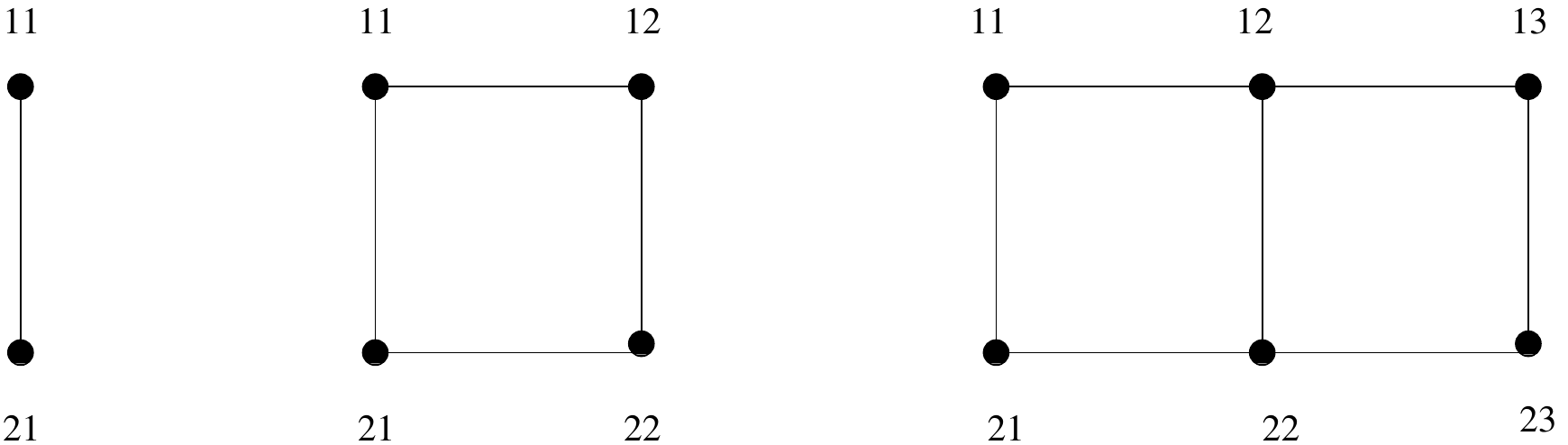}
\end{minipage}
\hspace{0.1cm}
\begin{minipage}[b]{0.4\linewidth}
\centering
   \includegraphics[scale=0.25]{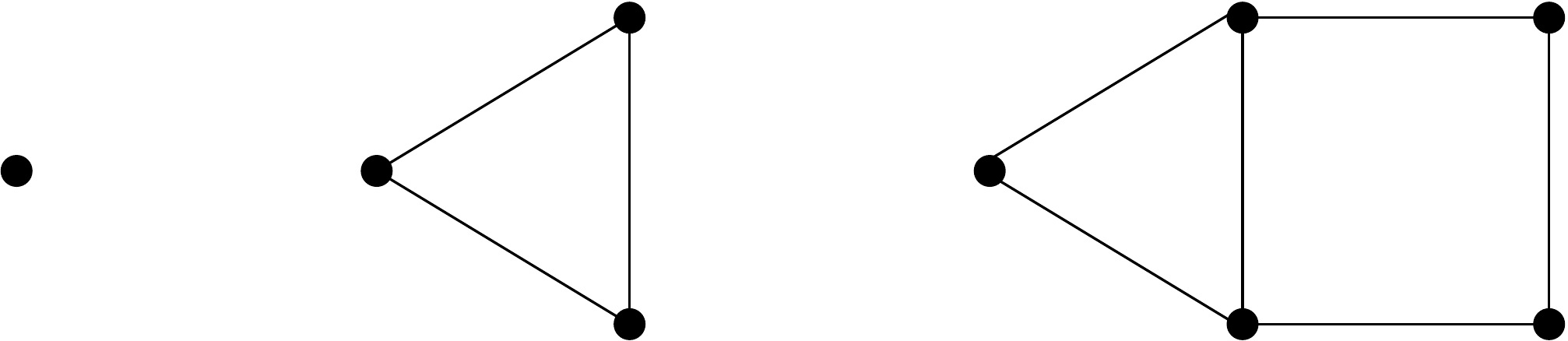}
\end{minipage}
 \caption{The first three terms are from the sequence $\{L_{2,n}\}$ 
          and the last three are from $\{Q_{2,n}\}$.}\label{fig:first_terms}
\end{center}
\end{figure}

 The initial conditions  to construct the recurrence
relations for the Tutte polynomial of the grid graphs $L_{2,n}$
are:

\begin{enumerate}

\item[i.] $T_{Q_{2,1}}(x,y) = 1$

\item[ii.] $T_{L_{2,1}}(x,y) = x$

\item[iii.] $T_{L_{2,2}}(x,y) = x^3 + x^2 + x + y$

\end{enumerate}

 The third condition is true because the graph $L_{2,2}$
is isomorphic to the cycle $C_4$, so we can use Equation~(\ref{eq:C_n}). The formula for the recurrence
relation to calculate the Tutte polynomial of
the graphs $L_{2,n}$ is:
\begin{eqnarray}\label{eq:L2n}
   T_{L_{2,n}}(x,y) &=& (x^2 + x + 1) T_{L_{2,n-1}}(x,y)+ yT_{Q_{2,n-1}}(x,y)\\
   T_{Q_{2,n}}(x,y) &=& (x + 1) T_{L_{2,n-1}}(x,y)+ y T_{Q_{2,n-1}}(x,y).
\end{eqnarray}

 From this, we get a linear recurrence of order two,
\[
    T_{L_{2,n}}-(x^2+x+1+y)T_{L_{2,n-1}}+ x^2 y T_{L_{2,n-2}}=0,
\]
 that is easy to solve. Thus, the general formula is
 \begin{equation}
   T_{L_{2,n}}(x, y) = \frac{a_1 \lambda_1+a_2}{\lambda_1-\lambda_2} \lambda_1^{n-2}
		   + \frac{a_1 \lambda_2+a_2}{\lambda_2-\lambda_1} \lambda_2^{n-2},
  \end{equation}
 for $n\geq 2$, where $a_1=y+x+x^2+x^3$, $a_2=-yx^3$ and $\lambda_1$ and $\lambda_2$ are
\[ 
 \frac{1}{2}\left[ (1+y+x+x^2)\pm \left( (y^2+2y(1+x-x^2)+(1+x+x^2)^2\right)^{1/2}\right]
\]

 A \emph{recursive family of graphs} $G_m$ is a sequence of graphs
 with the property that the Tutte polynomials $T_{G_m}(x, y)$ satisfy a linear 
 homogeneous recursion relation in which the coefficients are polynomials in $x$ and $y$
 with integral coefficients, independent of m, see~\cite{BDS72}. So, the sequence 
 $\{L_{2,n}\}$ is such a family 
 and in general,  $\{L_{k,n}\}$, for a fixed $k$, is a recursive family.

%%%%%%%%%%%%%%%%%%% graphic and representable matroids 
\subsubsection*{Graphic and representable matroids}
  Probably the most common class of matroids, when evaluations of the Tutte polynomial
  are considered, are graphic matroids. If $M$ is the graphic matroid of a graph $G$, 
  then deleting an element $e$ from $M$ 
  amounts to deleting the corresponding edge from $G$, and contracting $e$ to 
  contracting the edge. This is easy 
  to do in a small graph and by using Definition~\ref{def:recursive} you 
  can compute the Tutte polynomial. An 
  example is shown in Figure~\ref{Fig:recursive}
 \begin{figure}[hbtp]
 \begin{center}
 \includegraphics[scale=0.25]{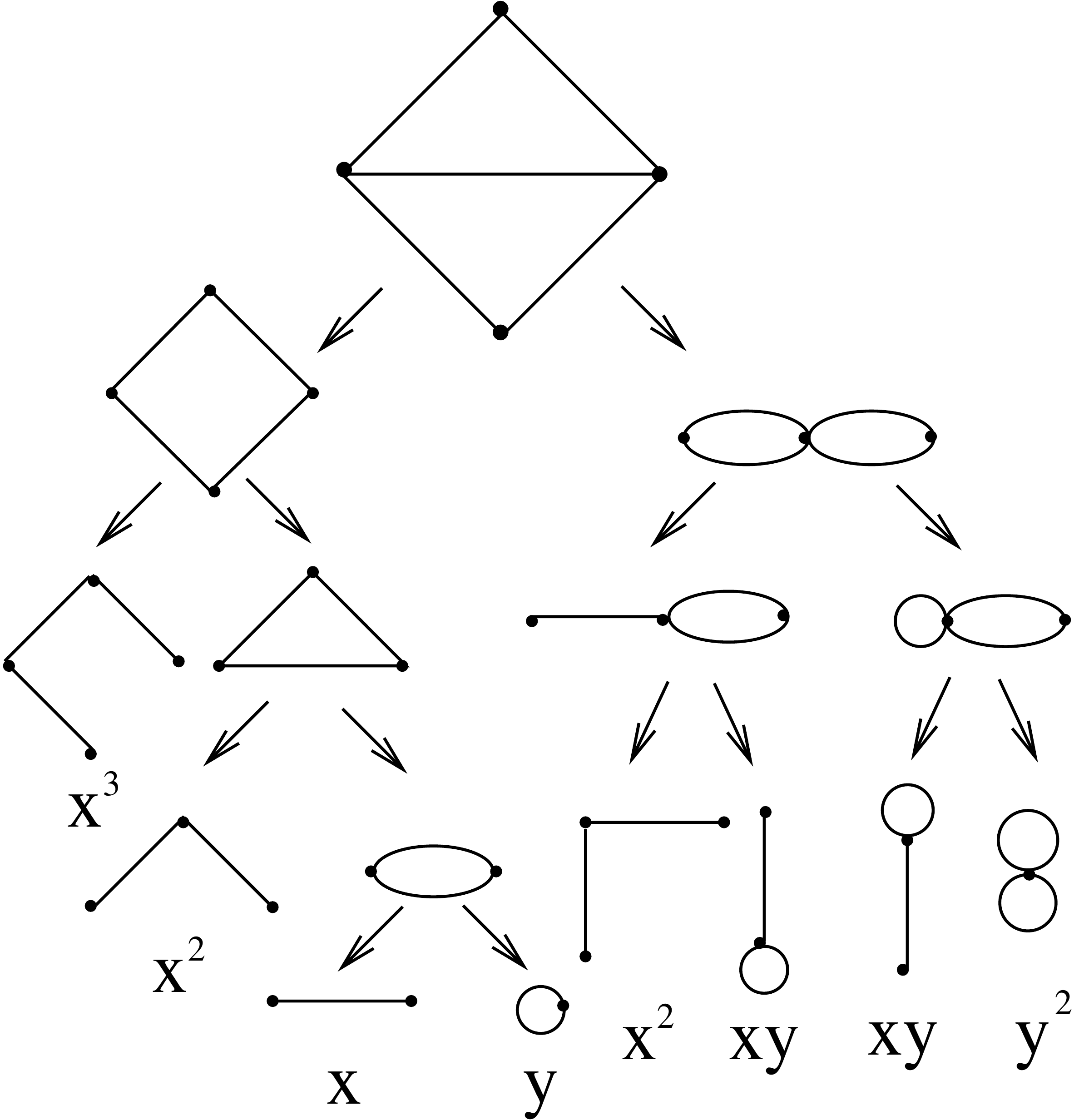}
 \caption{An example of computing the Tutte polynomial of a graphic matroid
   using deletion and contraction}\label{Fig:recursive}
 \end{center}
 \end{figure}

 As we mentioned computing the Tutte polynomial is not in general computationally
 tractable. However, for graphic matroids there are some resources to compute it 
 for reasonably sized graphs of about
 100 edges.  These include Sekine, Imai, and Tani~\cite{SIT95}, which
 provides an algorithm to implement the recursive Definition~\ref{def:recursive}. Common
 computer algebra systems such as Maple and Mathematica will compute
 the Tutte polynomial for very small graphs, 
 and there are also some implementations freely available on the Web,
 such as \url{http://ada.fciencias.unam.mx/~rconde/tulic/} by R. Conde
 or \url{http://homepages.mcs.vuw.ac.nz/~djp/tutte/} by G. Haggard and
 D. Pearce. 

  Similarly, given a matroid $M=(E,r)$,  representable over a field $\mathbb{F}$, you can compute its Tutte polynomial using  
 deletion and contraction. For any element $e\in E$, the matroids $M\setminus e$ and $M/e$ are easily computed from the
  representation of $M$, see~\cite{Oxl92}. This can be automatized and there are already computer programs that compute
 the  Tutte polynomial of a representable matroid. For our calculations  we have used the one given by Michael Barany, for more
 information about this program and how to use it, see~\cite{Bar}. Some of the calculations made in the 
 Section~\ref{Sec:small_matroids} were computed using this program. 

%%%%%%%%%%%%%%%%%%%%%%%% Activities in basis
%%%%%%%%%%%%%%%%%%%%%%%%%%%%%%%%%%%%%%%%%%%%%%%%

\subsection{Internal and external activity}
\label{subsec:internal_activity}
 Being a polynomial, it is natural to ask for the coefficients of 
 the Tutte polynomial. Tutte's original definition gives its 
 homonymous polynomial in terms of its coefficients by means of 
 a combinatorial interpretation, but before we give 
 the third definition of the Tutte polynomial, we
 introduce the relevant notions.

 Let us fix an ordering $\prec$ on the elements of $M$, say
 $E=\{e_1,\ldots,e_m\}$, where $e_i\prec e_j$ if $i<j$. Given a fixed
 basis $S$, an element $e$ is called \emph{internally active} if $e\in S$
 and it is the smallest edge with respect to $\prec$ in the only
 cocircuit disjoint from $S\setminus 
 \{e\}$. Dually, an element $f$ is \emph{externally active} if
 $f\not\in S$ and 
 it is the smallest element in the only circuit contained in $S\cup
 \{f\}$.  We 
 define $t_{ij}$ to be the number of bases with $i$
 internally active elements and $j$ externally active
 elements.  In \cite{Tut47,Tut54,Tut67} Tutte defined $T_{M}$ using these
 concepts. A proof of the equivalence with Definition~\ref{definition}
 can be found in~\cite{Bjo92}.

 \begin{definition} \label{def:activities}
 If $M=(E,r)$ is a matroid with a total order on its ground set, then
 \begin{equation}\label{tutte_activities}
   T_{M}(x,y)=\sum_{i,j}t_{ij}x^iy^j\;.
 \end{equation}
 In particular, the coefficients $t_{ij}$ are independent of the total order
 used on the ground set.
 \end{definition} 

%%%%%%%%%%%%%%%%%%%%% Uniform matroids again
\subsubsection*{Uniform matroids again}
 Our first use of~(\ref{tutte_activities}) is to get a slightly different 
 expression for the Tutte polynomial of uniform matroids. This time we get 
   \begin{equation}\label{eq:uniform_2}
          T_{U_{r,n}} (x,y)= \sum_{j=1}^{n-r}{n-j-1 \choose r-1}y^{j}+
                           \sum_{i=1}^{r}{n-i-1 \choose n-r-1}x^{i},
    \end{equation}
  when $0<r<n$, while $T_{U_{n,n}}(x,y)=x^{n}$ and  $T_{U_{0,n}}(x,y)=y^{n}$. 
  This can also be established by expanding~(\ref{eq:uniform_1}). 
  
%%%%%%%%%%%%%%%%%%%% Paving matroids
\subsubsection*{Paving matroids}

 Paving matroid were defined above. The importance of paving matroids is its abundance,
 meaning, most matroids of up to 9 elements are paving. This was checked
 in~\cite{MR08} and it has been conjecture in~\cite{MNWW11}
 that this is indeed true for all matroids,
 that is, when $n$ is large, the probability that you choose a paving matroid uniformly 
 at random among all matroids with up to $n$ elements is approaching 1.

 Now, in order to get a formula for the Tutte polynomial of paving matroids, we need
 the following definition from~\cite{Oxl92}. Given integers $k>1$ and  $m>0$, a collection  $\mathcal{T}=\{T_1,\ldots,T_k\}$ of subsets
 of a set $E$, such that each member of $\mathcal{T}$ has at least $m$
 elements and each $m$-element subset of $E$ is contained in a unique
 member of  $\mathcal{T}$, is called an $m$-partition of $E$. The elements of the
  partition are called \emph{blocks}. The
 following proposition is also from~\cite{Oxl92}.

 \begin{proposition}
   If $\mathcal{T}$ is an $m$-partition of $E$, then $\mathcal{T}$ is
   the set of hyperplanes of a paving matroid of rank $m+1$ on
   $E$. Moreover, for $r\geq 2$, the set of hyperplanes of every
   rank-$r$ paving matroid on $E$ is an $(r-1)$-partition of $E$.
 \end{proposition}
 
  Brylawski  gives the following proposition in~\cite{Bry72}, but we mentioned that his proof does not use activities. 
  
  \begin{proposition}\label{prop:paving}
  Let $M$ be a rank-$r$ matroid with $n$ elements and Tutte polynomial 
  $\sum_{i,j} t_{ij}x^i y^j$. Then, $M$ is paving iff $t_{ij}=0$ for all
   $(i,j)\geq (2,1)$. In addition if the $(r-1)$-partition of $E$ has $b_k$
    blocks of cardinality $k$, for $k=r-1,\ldots n$, then
  \begin{eqnarray*}
    t_{i0} &=& {n-i-1 \choose r-i} \text{ for all } i\geq 2;\\
    t_{10} &=& \sum_{k=0}^{\infty} {r-2+k \choose r-2}b_{k+r-1} + {n-2 \choose r-1} -{n \choose r-1};
  \end{eqnarray*}   
  and for all $j>0$,
  \begin{eqnarray*}  
    t_{1j} &=& \sum_{k=0}^{\infty} {r-2+k \choose r-2}b_{k+j+r-1},  \text{ and }\\
    t_{0j} &=& {n-j-1 \choose r-1} - \sum_{k=0}^{\infty} {r-1+k \choose r-1}b_{k+j+r-1}.
  \end{eqnarray*}   
  \end{proposition}
  
  Note that when $M$ is the uniform matroid $U_{r,n}$, $M$ is a paving matroid
  with $(r-1)$-partition all subsets of size $r-1$. Thus the above formula give
  us equation~(\ref{eq:uniform_2}).
  
\begin{figure}[h!]
\begin{center}
\includegraphics[scale=0.5]{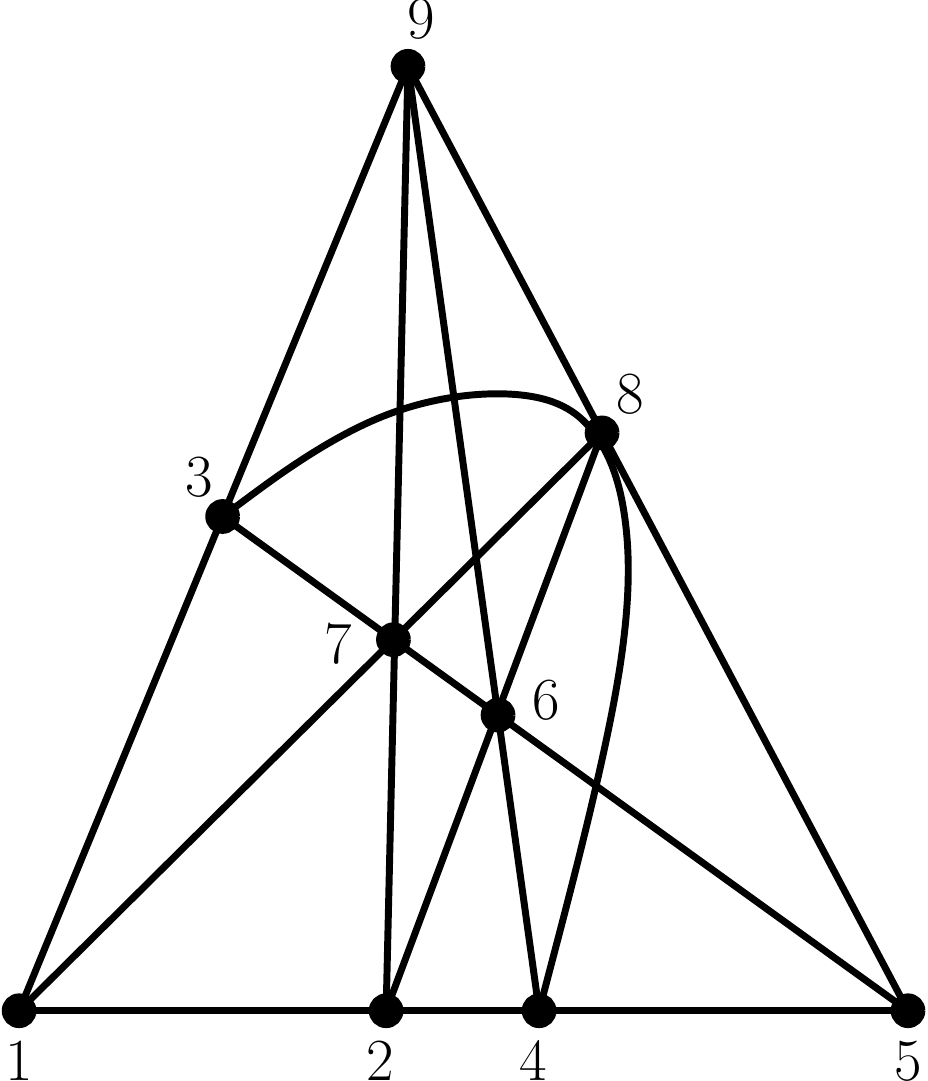}
\caption{The matroid $R_9$}\label{fig:R9}
\end{center}
\end{figure}
  
  As a example consider the matroid $R_9$ with geometric representation given in 
  Figure~\ref{fig:R9}. This matroid is paving but not sparse paving. The blocks
  correspond to lines in the representation. There are 7 blocks of size 3, 
  2 blocks of size 4 and 3 blocks of size 2, corresponding to trivial lines that
  do nor appear in Figure~\ref{fig:R9}. Using the above formulas we get that the Tutte 
  polynomial of $R_9$ is 
  \begin{equation*}
    T_{R_9}(x,y)=x^3 +6x^2  +8x+ 11x y + 2 x y^2+ 8 y+13y^2+10y^3+6y^4+3y^5+y^6.
  \end{equation*}

%%%%%%%%%%%%%%%%%%%%% Catalan matroids
\subsubsection*{Catalan matroids}

 Our final example in this subsection is a family of matroids that although simple, 
 it has a surprisingly natural combinatorial definition together
 with a nice interpretation of the internal and external activity. 
 
 Consider an alphabet constituted by the letters
$\{N,E\}$. The {\em length} of a word $w=w_1 w_2 ... w_n$ is $n$,
the number of letters in $w$. Every word can be associated with a
path on the plane $\mathbb{Z}^2$, with a initial point $A=(m_1, n_1)$
and a final point $B=(m_2, n_2)$, $m_1 < m_2$ and $n_1 < n_2$. The
letter $N$ is identified with a north step and the letter $E$ with
a east step, such that the first step of $A=(0,0)$ to $B$ will be
$N=(0,1)$ or $E=(1,0)$.

 Let $A$ and $B$ be a couple of fixed points of $\mathbb{Z}^2$. A
{\em lattice path} is a path in $\mathbb{Z}^2$ from $A$ to $B$ using
only steps $N$ or $E$. The lattice paths from $A$ to $B$ have
$m$ steps $E$ and $n$ steps $N$. Let $P$ and $Q$ be two lattice
paths, with initial points $(x_0,y_P)$ and $(x_0,y_Q)$ of $P$ and
$Q$ respectively. If $y_P < y_Q$ for every $x_0$ in $[m_1,m_2]$;
then the set of lattice paths in this work are lattice paths from
$A$ to $B$ bounded between $P$ and $Q$, that we called $PQ$-bounded lattice paths.

The lattice paths $P$ and $Q$ can be substituted by lines. If $P$ is
the line $y=0$, $Q$ the line $y=x$ and $m=n$, then the number of
$PQ$-bounded lattice paths  from $(0,0)$ to $(n,n)$ is the $n$-th
Catalan number:
\begin{center}
$c_n = \frac{1}{n+1} \left( \begin{array}{c}
         2n \\ n
         \end{array} \right).$
\end{center}

 Denote by $[n]$  the set $\{1,2,...,n\}$. Consider the
 word $w$=$w_1$, $w_2$, $...$, $w_{m+n}$ that corresponds to a $PQ$-bounded lattice path. We can 
 associate to $w$ a subset $X_w$ of $[m+n]$ by
 defining that $i \in X_w$ if $w_i=N$. $X_w$ is the support set of the
 steps $N$ in $w$. The family of support sets $X_w$ of  $PQ$-bounded lattice paths
 are the set of bases of a transversal matroid, denoted $M[P,Q]$ with ground set
 $[m+n]$, see~\cite{BMN03}. 

 If the bounds of the lattice paths are the lines $y=0$ and
 $y=x$, the matroid is denoted by $M_n$ and is named {\em Catalan
 matroid}.  Note that $M_n$ has a loop that corresponds to the label
 $1$ and a coloop that corresponds to the label $2n$, and 
 that $M_n$ is autodual.

 The fundamental result to compute the Tutte polynomial of Catalan matroids is the following, see~\cite{BMN03}, where here, for a basis $B$, we denote by $i(B)$ the number of internally active elements and by $e(B)$ the externally active elements.
\begin{proposition}%\label{p1:1}
 Let $B \in \cal{B}$ be a basis of $M[P,Q]$ and let $w_B$ be the $PQ$-bounded lattice path 
 associated with $B$. Then $i(B)$ is the number of
 times $w_B$ meets the upper path $Q$ in a north step and $e(B)$ is
 the number of times $w_B$ meets the lower path $P$ in an east
 step.
\end{proposition}

Then the Tutte
polynomial of the Catalan matroids $M_n$ for $n > 1$,
is:
\begin{equation}
 T_{M_n}(x, y)=\sum_{i,j > 0} \frac {i+j-2}{n-1} \left( \begin{array}{c}
         2n-i-j-1 \\ n-i-j+1
         \end{array} \right) x^i y^j.
\end{equation}

Note that the coefficient of $x^{i} y^{j}$ in the Tutte
polynomial of the matroid $M_n$ depends only on $n$ and the sum
$i+j$.

 In the Figure~\ref{M_3} are shown the bases of $M_3$  with its internally and external activities. 
 Thus the Tutte polynomial of $M_3$ is
 \[ T_{M_3}(x,y)= x^3 y + x^2 y + x^2 y^2 + x y^2 + x y^3. \]

%%%%%%  Aumentar tamaño de letras en la figura !!!!!!!

\begin{figure}
  \begin{center}
  \includegraphics[scale=0.38]{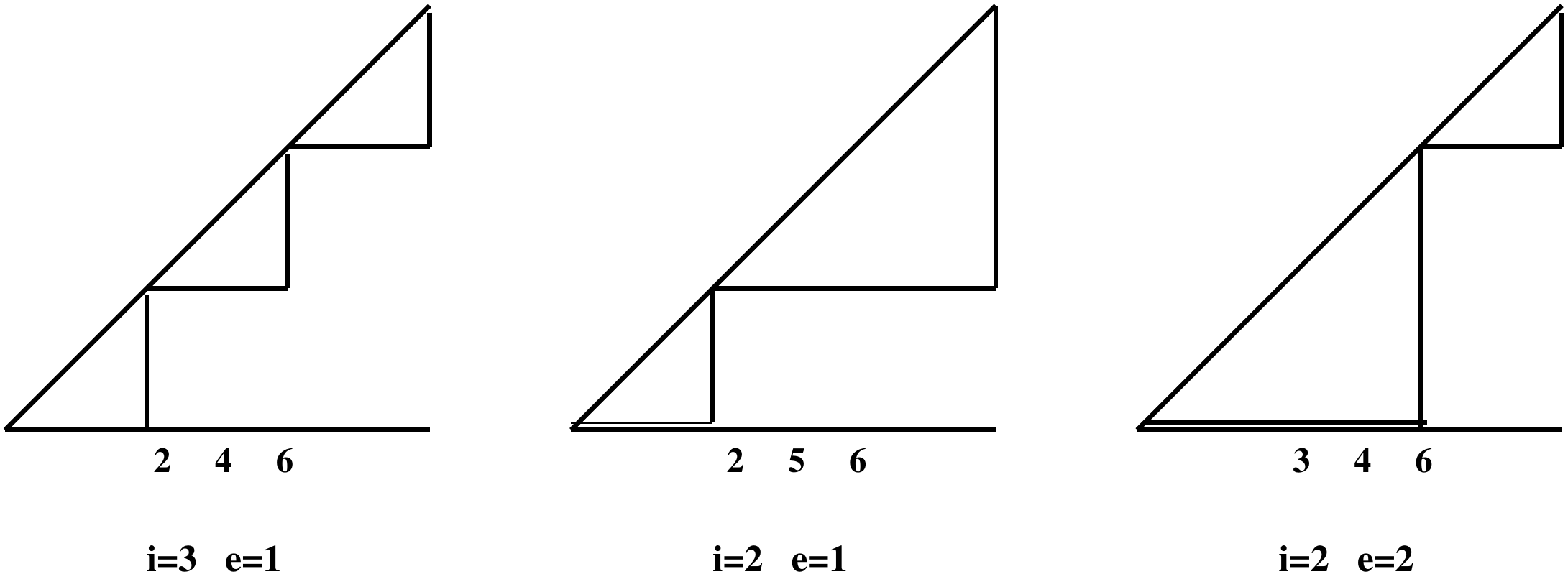} \hspace{7.7 cm}
   \includegraphics[scale=0.38]{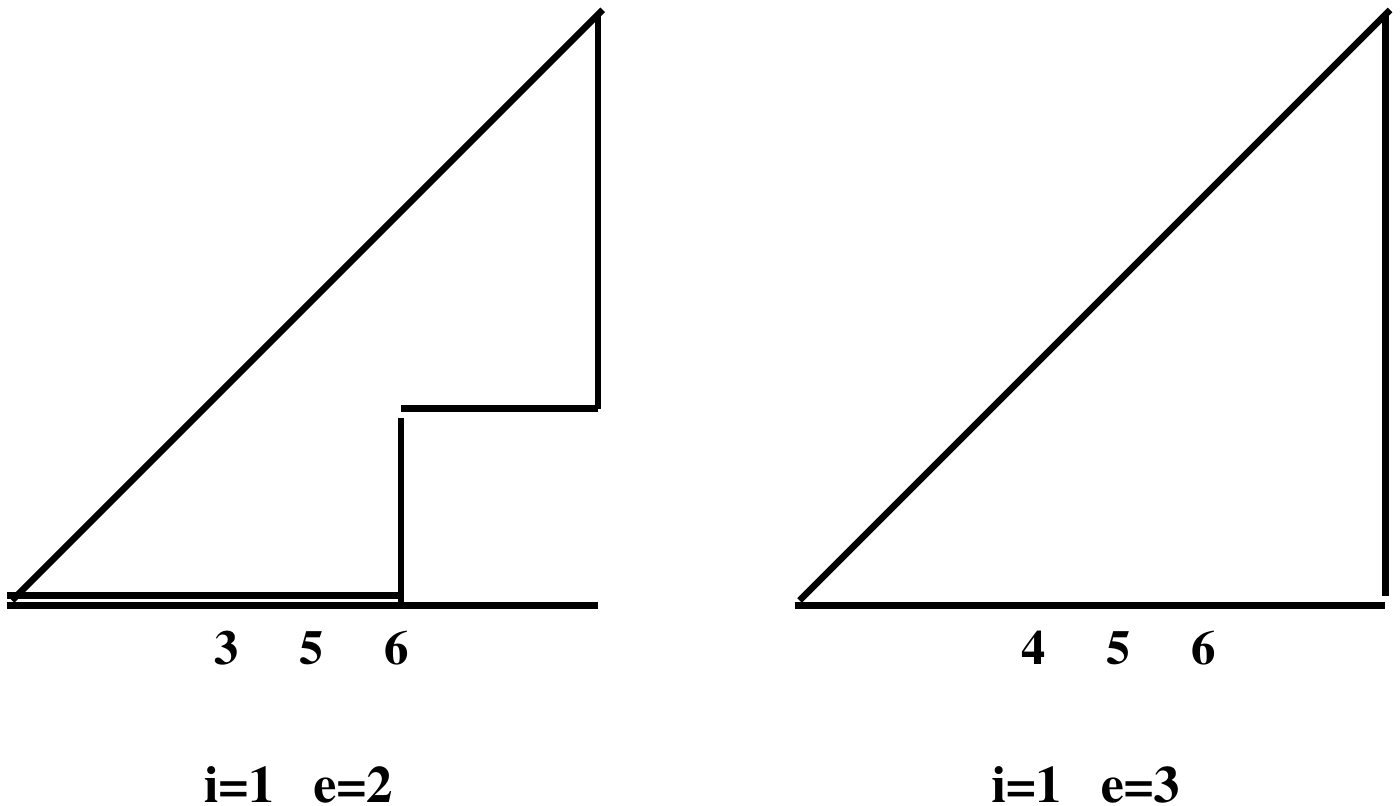}
  \caption{Bases of $M_3$ with their internal and external activity.}\label{M_3}
 \end{center}
\end{figure}

\section{Techniques}

 We move now to present three techniques that are sometimes useful to compute
 the Tutte polynomial, when the initial trial with the above
 definitions are not successful.

%%%%%%%%%%%%%% Equivalent polynomials: The coboundary polynomial
%%%%%%%%%%%%%%%%%%%%%%%%%%%%%%%%%%%%%%%%%%%%%%%%%%%%%%%%%%

\subsection{Equivalent polynomials: The coboundary polynomial}

 For some families of matroids and graphs it is easier to compute
 polynomials that are equivalent to the Tutte polynomial. One such
 polynomial is  Crapo's coboundary polynomial, see also~\cite{Bry79}.  
 \begin{equation}\label{coboundary}
   \bar{\chi}_{M}(\lambda,t)=\sum_{X\in \mathcal{F}_{M}} t^{|X|} \chi_{M/X}(\lambda)
\end{equation}
 where $\mathcal{F}_{M}$ is the set of flats of $M$ and $\chi_{N}(\lambda)$ is the characteristic polynomial of the matroid $N$. The characteristic
 polynomial of a matroid $M$  is defined by
 \begin{equation}\label{characteristic}
   \chi_{N}(\lambda)=\sum_{A\subseteq E} (-1)^{|X|} \lambda^{z(A)}.
\end{equation}

  The Tutte and the coboundary polynomial are related as follows
\begin{equation}\label{Tutte_to_Chi}
  T_{M}(x,y)=\frac{1}{(y-1)^{r(E)}}\;\bar{\chi}_{M}((x-1)(y-1),y).
\end{equation}
\begin{equation}\label{Chi_to_Tutte}
   \bar{\chi}_{M}(\lambda,t)=(t-1)^{r(E)}\;T_{M}(\frac{\lambda+t-1}{t-1},t).
\end{equation}

 When $M$ is a graphic matroid $M(G)$ of a connected graph $G$, the
 characteristic polynomial 
 is equivalent to the chromatic polynomial, that we denote $\chi_{G}$.
 \[   \chi_{G}(\lambda)=\lambda\;\chi_{M(G)}(\lambda).\]

 And also, in this case, the coboundary is the bad-colouring
 polynomial. The \emph{bad colouring polynomial} is the generating function
 \[
    B_{G}(\lambda, t)= \sum_{j}{b_j(G; \lambda) t^j},
 \]
 where $b_{j}(G; \lambda)$ is the number of $\lambda$-colourings of
 $G$ with exactly $j$ bad edges. Note that when we set
 $t=0$ we obtain the chromatic polynomial. Then, the polynomials are
 related as  follows: 
 \begin{equation}\label{Bad-coboundary}
    B(G; \lambda, t)=\lambda\;\bar{\chi}_{M(G)}(\lambda,t).
 \end{equation}

%%%%%%%%%%%%%%%%%%%% q-cone
\subsubsection*{The $q$-cone}
 We base this section on the work of J. E. Bonin and  H. Qin, see~\cite{BQ01}.  
 \begin{definition}
  Let $M$ be a rank-r simple matroid representable over GF(q). A matroid $N$ is a $q$-cone
  of $M$ with base $S$ and apex $a$ if
  \begin{enumerate}
	\item the restriction $PG(r, q)|S$ of $PG(r, q)$ to the subset $S$ is isomorphic to $M$,
	\item The point $a \in PG(r, q)$ is not contained in the closure of $S$, $CL_{PG(r,q)}(S)$, in $PG(r, q)$, and
	\item The matroid $N$ is the restriction of $PG(r, q)$ to the set $\bigcup_{s\in S} Cl_{PG(r,q)}(\{a,s\})$.
  \end{enumerate}	
 \end{definition}
 
  That is, one represents $M$ as a set $S$ of points in $PG(r, q)$ and  constructs $N$ by restricting $PG(r, q)$ to the set of points on the lines joining the points of $S$ to the fixed point $a$ outside the hyperplane of 
$PG(r, q)$ spanned by $S$. The basic result is by Kung who proved in~\cite{Kun96}.
  \begin{theorem}
   For every $q$-cone N of a rank-$r$ simple matroid M, we have
   \[
       \chi_{N}(\lambda)=(\lambda-1)q^{r} \chi_{M}(\lambda/q).
   \]
   \end{theorem}
  To extend this result to $\bar{\chi}$ the authors in~\cite{BQ01} identify the flats $F$ of $N$ and the contractions $N/F$ of $N$ by these flats. They prove the following formula
  \begin{equation}
    \bar{\chi}_{N}(\lambda,t)=t \bar{\chi}_{M}(\lambda,t^{q})+
			       q^{r}(\lambda-1)\bar{\chi}_{M}(\lambda/q,t).
  \end{equation}
 By using equation~(\ref{Tutte_to_Chi}) we get a formula for the Tutte polynomial of $N$ in terms of the Tutte polynomial of $M$. 
   \begin{theorem}
   If $M$ is a rank-$r$ matroid representable over $GF(q)$ and $N$ is a $q$-cone of $M$, then
   \[
   T_{N}(x,y) = 
   \frac{y(y^{q}-1)^{r}}{(y-1)^{r+1}}
              T_{M}(\frac{(x-1)(y-1)}{y^{q}-1}+1,y^{q}) +
   \frac{q^{r}(xy-x-y)}{y-1}T_{M}(\frac{x+1}{q}+1,y).           
   \]
  \end{theorem}
  Thus, for example, the $2$-cone of the 3-point line is the matroid $PG(2,2)$.
  The Tutte polynomial of the former matroid is just $x^2+x+y$, then by using the above formula we get the Tutte polynomial of $PG(2,2)$ to be
  \[
    x^3 +4x^2 +3x +7xy  +3y +6y^2 +3y^3 +y^4.
    \]
 As the $q$-cone of $PG(r-1,q)$ is $PG(r,q)$, this method can be used to compute the Tutte polynomial of any $PG(r,q)$, but we do this later using the coboundary polynomial directly.

%Complete graphs
\subsubsection*{Complete graphs}
 Given the apparent simplicity of many of the formulas for invariants in complete
 graphs, like the number of spanning trees or acyclic orientations, it is not
  so straightforward to compute the whole Tutte polynomial of complete graphs; many 
  researches, however, have tried with different amounts of success. The amount 
  of frustration after failing to compute this seemingly innocuous invariant of 
  many of our colleagues was our original motivation for writing this survey paper. 

 Using the exponential formula, see Stanley\cite{Sta99}, we can give
 an exponential generating function for the Tutte polynomial of the
 complete graphs. Let us denote the vertices of  $K_{n}$ by $V$ and  by 
 $B_{n}$ its bad-colouring.
 To compute $B_{n}$, observe that any $\lambda$-colouring
 partitions the vertices $V$  into $\lambda$ color classes of  subsets
 $V_{i}$ of vertices each of cardinality $n_i$, for $i=1\ldots \lambda$. So,
 we have that
 $n_1+\ldots+n_\lambda=n$. The number of bad edges with both ends in the set
 $V_i$  is $t^{n_i \choose 2}$. Thus, by 
 the exponential formula we get the following formula 
 \begin{equation}  \label{Formula_Kn_B_n}
  \left(\sum_{n \in \mathbb{N}} t^{n \choose 2} 
               \frac{u^n}{n!} \right)^{\lambda}
  =  1+ \sum_{n\geq 1}  B_{n}(\lambda,t)\frac{u^n}{n!}.
\end{equation}
 
 Now, we can use  equation~(\ref{Bad-coboundary}) to get
 \begin{equation} \label{Formula_Kn_X_n}
  \left(\sum_{n \in \mathbb{N}} t^{n \choose 2} 
               \frac{u^n}{n!} \right)^{\lambda}
  =  1+\lambda \sum_{n\geq 1}  \bar{\chi}_{n}(\lambda,t)\frac{u^n}{n!}.
\end{equation}

 Let $T_n(x,y)$ be the Tutte polynomial of $K_n$.  Tutte in
 \cite{Tut67} and  Welsh in
 \cite{Wel96} give the following exponential generating function for
 $T_n(x,y)$ that  follows  from the previous equation and
 equation~(\ref{Tutte_to_Chi}). 
 \begin{equation}
   \label{Formula_Kn}
    \left( \sum_{n\geq
     0}{y^{{n}\choose{2}}\frac{u^{n}}{n!}}\right)^{(x-1)(y-1)}= 
     1+(x-1)\sum_{n\geq 1}{(y-1)^{n}T_n(x,y)\frac{u^{n}}{n!}}.
 \end{equation}

   Even that the previous formulas seem difficult to handle,
  equation~(\ref{Formula_Kn_X_n}) it is quite easy to use in Maple or
  Mathematica. For example, the following program in Maple can compute
   $T_{30}(x,y)$ in no time. 
  \begin{verbatim}
  Coboundary:= proc(n,q,v) local i,x,g;
                g:=x->(add(v^(i*(i-1)/2)*x^i/i!,i=0..n))^q;
                simplify(eval(diff(g(x),x\$n),x=0)/q);end proc;
  T:= proc(n,x,y)
      simplify( (1/((y-1)^(n-1)))*subs({q=(x-1)*(y-1),v=y},
      Coboundary(n,q,v)) ) end proc;
 \end{verbatim}

  By computing $T_{5}(x,y)$ we obtain
  \begin{equation*}
  \begin{split}
  T_{5}(x,y) & ={y}^{6}+4\,{y}^{5}+{x}^{4}+5\,x{y}^{3}+10\,{y}^{4}+6\,
             {x}^{3}+10\,{x}^{2}y+15\,x{y}^{2}+15\,{y}^{3}\\
             & \quad +11\,{x}^{2}+20\,xy+15\,{y}^{2}+6\,x+6\,y.
  \end{split}
  \end{equation*}

%Complete bipartite graphs
 \subsubsection*{Complete bipartite graphs}

   It is just natural to use this technique for complete bipartite graphs 
   $K_{n,m}$ with similar results as above. This time let the vertex set 
   be $V_1\cup V_2$ and  denote by $B_{n,m}$  its bad-colouring polynomial.
  To compute $B_{n,m}$, observe that any $\lambda$-colouring
  partitioned the vertices in subsets $V_1^{(i)}\cup V_2^{(i)}$ each of
  cardinality $n_i m_i$.  The number of
  the bad edges with both ends colour $i$ is $n_i m_i$. Thus, by the
  exponential formula we get the following.
  \begin{equation} \label{Formula_Knm_B_nm}
   \begin{split}
   \left(\sum_{(n,m)\in \mathbb{N}^2} t^{nm} 
                \frac{u^n}{n!}\frac{u^{m}}{m!} \right)^{\lambda}
   &=  1+ \sum_{\substack{
                    (n,m)\in\mathbb{N}^2\\
                    (n,m)\neq (0,0)}}
   B_{n,m}(\lambda,t)\frac{u^n}{n!}\frac{v^{m}}{m!}\\
   &= 1+\lambda \sum_{\substack{
                    (n,m)\in\mathbb{N}^2\\
                    (n,m)\neq (0,0)}}
   \bar{\chi}_{n,m}(\lambda,t)\frac{u^n}{n!}\frac{v^{m}}{m!}.
  \end{split}
 \end{equation}

  Thus, a formula for the Tutte polynomial of the bipartite
  complete graph can be found in a similar way as before.
  Let $T_{n,m}(x,y)$ be the Tutte polynomial of $K_{n,m}$. The
  following formula is from Stanley's book \cite{Sta99}, see also~\cite{MR05}

 \begin{equation} \label{Formula_Knm}
  \begin{split}
    \left(
     \sum_{(n,m)\in \mathbb{N}^2} y^{nm}\frac{u^{n}}{n!}\frac{v^{m}}{m!}
    \right)^{(x-1)(y-1)}
     &=1+(x-1)\\
     &\quad\sum_{\substack{
                    (n,m)\in\mathbb{N}^2\\
                    (n,m)\neq (0,0)}}
      (y-1)^{nm}T_{n,m}(x,y)\frac{u^n}{n!}\frac{v^{m}}{m!}.
  \end{split}
 \end{equation} 

  As before, equation~(\ref{Formula_Knm_B_nm})  is quite easy to use
  in Maple to compute $T_{n,m}(x,y)$ for small values of $n$ and
  $m$. In this way we get the following.
  \begin{equation}\label{tute_K33}
  \begin{split}
  T_{3,3}(x,y) &=
  {x}^{5}+4\,{x}^{4}+10\,{x}^{3}+9\,{x}^{2}y+11\,{x}^{2}+6\,x{y}^{2}\\
  &\quad +15\,xy+5\,x+{y}^{4}+5\,{y}^{3}+9\,{y}^{2}+5\,y.
  \end{split}
  \end{equation} 
%%  For example, to compute $T_{3,3}(x,y)$ you need to do the
%%   following.  
%%   \begin{verbatim}
%%    Coboundary:= proc(n,m,q,v) local i,j,g; 
%%                 g:=(x,y)->(expand(add(add(v^(i*j)*x^i/i!,i=0..n)
%%                                        *y^j/j!, j=0..m)) )^q;
%%                 simplify(eval(eval(diff(g(x,y),y$m,x$n),
%%                                    x=0),y=0)/q);end proc:

%%   T:= proc(n,m,x,y) 
%%       simplify( (1/((y-1)^(n+m-1))) * subs({q=(x-1)*(y-1),v=y},
%%       Coboundary(n,m,q,v))); end proc:
%%  \end{verbatim}

%%   By computing T(3,3,x,y) we get 

%Proyective geometries and afine geometries
\subsubsection*{Projective geometries and affine geometries}

 The role played by complete graphs in graphic matroids is taken by projective
 geometries in representable matroids. Herein lies the importance of projective geometries.
 Even though a formula for their Tutte polynomial has been known for a while, and has 
 been discovered at least twice, not much work has been done in the actual 
 combinatorial interpretations for the value of the different evaluations of the Tutte 
 polynomial. 
 
 For this part we follow Mphako \cite{Mph00}, see also~\cite{BR05}. For all  non-negative integers
 $m$ and $k$  we define the {\em Gaussian coefficients} as
 \begin{equation}
   \left[ {m \atop k}\right]_q=\frac{(q^m-1)(q^m-q)\cdots(q^m-q^{k-1})}
                                  {(q^k-1)(q^k-q)\cdots(q^k-q^{k-1})}.
 \end{equation}
 Note that $\left[ {m \atop 0}\right]_q=1$ since by convention an
  empty product is 1 and  $\left[ {0 \atop k}\right]_q=0$. 

\begin{figure}
\begin{center}
 \includegraphics[scale=0.25]{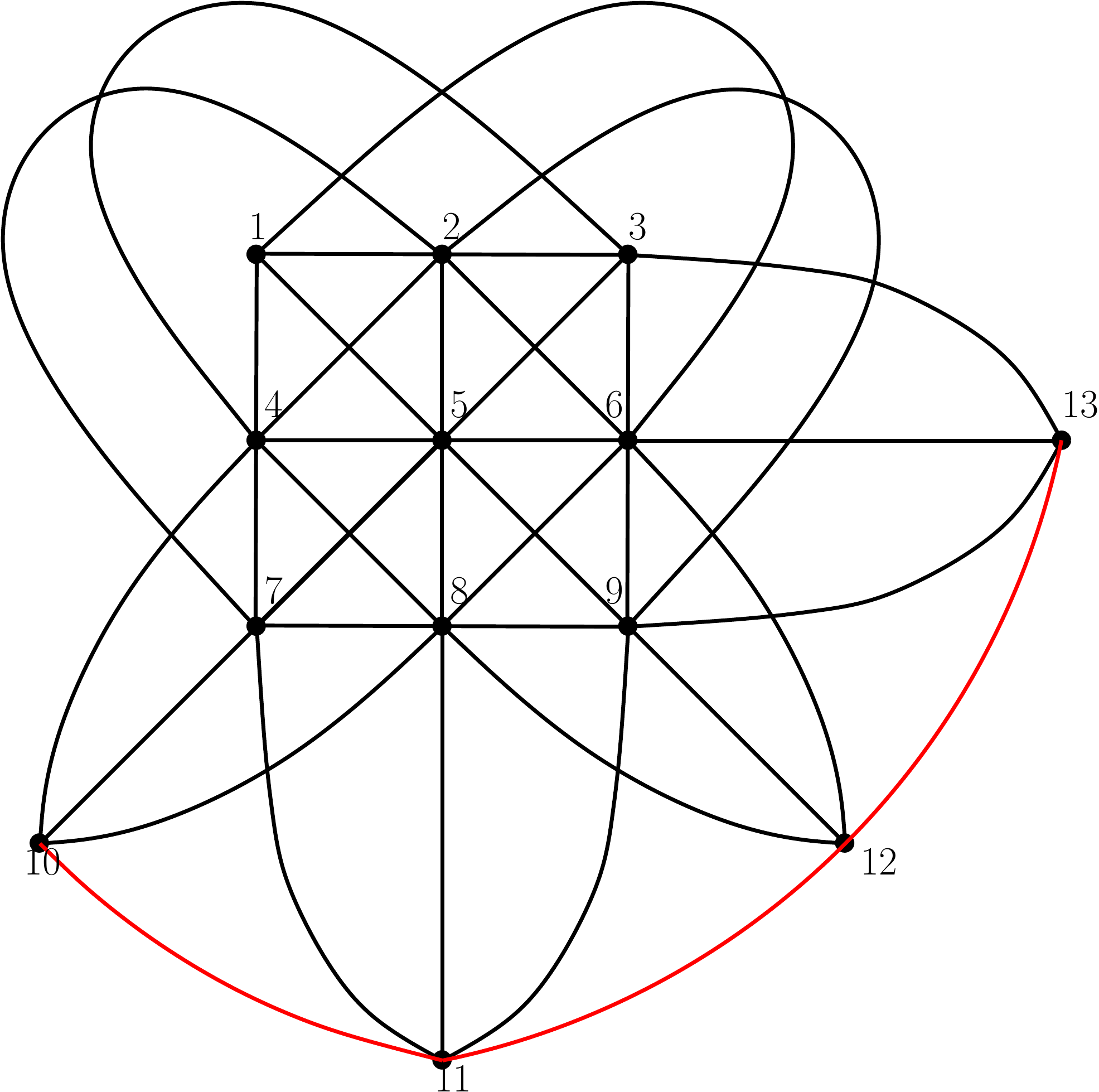}
 \caption{The proyective plane $PG(2,3)$ with a marked hyperplane}
 \label{fig:pg23}
\end{center}
\end{figure}

  Let $PG(r-1,q)$ be the $(r-1)$ dimensional projective geometry over $GF(q)$. 
 Then, as a matroid, it has rank $r$ and $\left[ {r \atop 1}\right]_q$
 elements. Also, every rank-$k$ flat $X$ is isomorphic to $PG(k-1,q)$ and
  the simplification of $M/X$ is isomorphic to  $PG(r-k-1,q)$. The
 number of  rank-$k$ flats  is $\left[ {r \atop
 k}\right]_q$. The characteristic polynomial of $PG(r-1,q)$ is known
 to be, see~\cite{BO92},
 \[
   \chi_{PG(r-1,q)}(\lambda)=\prod_{i=0}^{r-1}(\lambda-q^{i})
 \]   
Thus, using equation~(\ref{coboundary}) we get
 \begin{equation}\label{eq:projective}
   \bar{\chi}_{PG(r-1,q)}(\lambda,t)=\sum_{k=0}^{r}
             t^{\left[ {k \atop 1} \right]_q}  %num points
              \left[ {r \atop k} \right]_q       %num k-flats
              \prod_{i=0}^{r-k-1}(\lambda-q^{i}).  %char poly of PG(r-k-1,q)
 \end{equation}

%PG(2,3)
%\subsubsection*{The matroid $PG(2,3)$}
 The matroid $PG(2,3)$, has a geometric representation given in Figure~\ref{fig:pg23}, and it is isomorphic to the unique Steiner system $S(2,4,13)$. The matroid is paving but not sparse paving and has the following representation over $GF(3)$.
\[
\bordermatrix[{[]}]{ & 1 & 2 & 4 |& 3 & 5 & 6 & 7 & 8 & 9 & 10 & 11 & 12 & 13\cr
                     & 1 & 0 & 0 |& 1 & 2 & 2 & 1 & 1 & 0 & 0  & 1  & 1  & 2\cr
                     & 0 & 1 & 0 |& 1 & 1 & 1 & 0 & 1 & 1 & 1  & 0  & 1  & 1 \cr
                     & 0 & 0 & 1 |& 0 & 1 & 2 & 1 & 2 & 1 & 2  & 2  & 1  & 0 \cr}.
\]
 To compute its Tutte polynomial we could use the program in~\cite{Bar} or use the above formula to get
  \begin{equation*}
   \bar{\chi}_{PG(2,3)}(\lambda,t)=  (\lambda-1)(\lambda-3)(\lambda-9)+
				    13 t(\lambda-1)(\lambda-3)+
				    13 t^{4} (\lambda-1)+
			 	    t^{13}.
 \end{equation*}

Finally the Tutte polynomial of $PG(2,3)$ is obtained from the coboundary by a substitution $\lambda=(x-1)(y-1)$ and $t=y$ and by multiplying by the factor $1/(y-1)^3$
\begin{eqnarray}
 T_{PG(2,3)}(x,y) & = & x^3+10x^2+13xy^2+26xy+16x+16 y + 32 y^2 + 36 y^3 \nonumber \\
                  &   & + 28 y^4 + 21 y^5 + 15 y^6 + 10 y^7 + 6 y^8 + 3 y^9 +  y^{10}.
\end{eqnarray}

 Every time we have a projective geometry $PG(r-1,q)$ we can get an \emph{affine geometry}, $AG(r-1,q)$, simply by deleting all the points in a hyperplane of  $PG(r-1,q)$. For example, by deleting all the red points from $PG(2,3)$ in
Figure~\ref{fig:pg23} we get $AG(2,3)$ with geometric representation in Figure~\ref{fig:ag23}.
\begin{figure}%[h]
\begin{center}
 \includegraphics[scale=0.4]{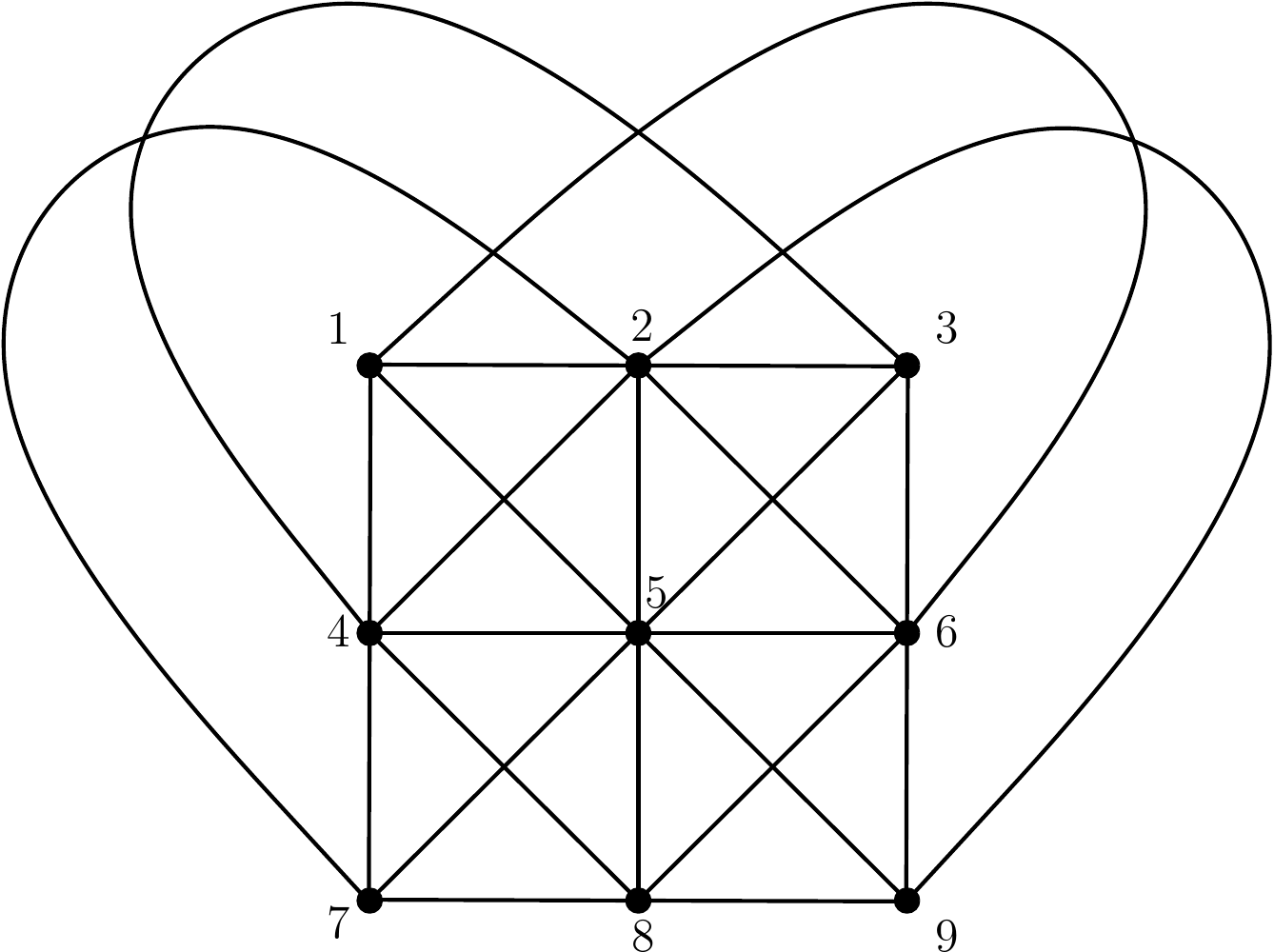}
 \caption{Afine plane $AG(2,3)$}
 \label{fig:ag23}
\end{center}
\end{figure}

 The situation to compute the Tutte polynomial is very similar for the $r$-dimensional affine geometry over $GF(q)$, 
 $AG(r,q)$.  It has rank $r+1$  and  $q^{r}$ elements. Any flat of $X$ rank $k+1$ is isomorphic to  $AG(k,q)$ and 
 the simplification  of  $AG(r,q)/X$ is isomorphic to $PG(r-k-1,q)$.  The number of  rank-$k$ flats  is $q^{r-k}  
 \left[ {r \atop k} \right]_q $. The characteristic polynomial of $AG(r,q)$ is known to be, see~\cite{BO92},
 \[
   \chi_{AG(r,q)}(\lambda)=(\lambda-1) \sum_{k=0}^{r}{ (-1)^{k} \lambda^{r-k} \prod_{i=0}^{k-1}(q^{r-i}) }
 \]   
 Now, using the result of Mphako we can compute the coboundary polynomial of $AG(r,q)$.
 \begin{equation}\label{affine}
   \bar{\chi}_{AG(r,q)}(\lambda,t)= \chi_{AG(r,q)}(\lambda)+
				\sum_{k=0}^{r}
        			     t^{q^{k}}  %num points
				     q^{r-k} \left[ {r \atop k} \right]_q       %num k-flats
			             \prod_{i=0}^{r-k-1}(\lambda-q^{i}) .  %char poly of PG(r-k-1,q)
 \end{equation}

 The affine plane $AG(2,3)$ is isomorphic to the unique Steiner triple system $S(2,3,9)$, as every line contains exactly 3
 points and every pair of points is in exactly one line. $S(2,3,9)$  is a sparse paving matroid with 12 circuit-hyperplanes
 so by using equation~(\ref{eq:sparse_paving}) we could compute its Tutte polynomial. However, we use the above formula to
 obtain the same result. 
 \begin{equation*}
   \bar{\chi}_{AG(2,3)}(\lambda,t)= (\lambda-1)(\lambda^{2}-8\lambda +16)+
				    9 t (\lambda-1)(\lambda-3)+
				    12 t^{3} (\lambda-1)+
			 	    t^{9}.
 \end{equation*}
 Again, the Tutte polynomial of $AG(2,3)$ is obtained from the coboundary by a substitution $\lambda=(x-1)(y-1)$ and $t=y$ and by multiplying by the factor $1/(y-1)^3$
\begin{equation*}
 T_{AG(2,3)}(x,y)= x^3+6x^2+12xy+9x+9y+15y^2+10y^3+6y^4+3y^5+y^6.
\end{equation*}

%%%%%%%%%%%%%%%%%%%% Transfer method
%%%%%%%%%%%%%%%%%%%%%%%%%%%%%%%%%%%%%%%%%%%%%%
\subsection{Transfer-matrix method}
  Using formula~(\ref{eq:expansion}) quickly becomes prohibitive  
  as the number of states grow exponentially with the size of the matroid.
  You can get around this problem when you have a family of graphs that are
   constructed using a simple graph
  that you repeat in a path-like fashion; the bookkeeping of the
  contribution of each state can be done with a matrix,  the update can be done
   by matrix multiplication after the graph grows a little more.
  This is the essence of our second method.
   
 The theoretical background of the transfer-matrix method, taken from
 \cite{Sta97}, is described below.

 A {\it directed graph} or {\it digraph} $\vec{G}$ is a triple $(V, E,
 \phi)$, where $V=\{v_1,\ldots, v_p\}$ is a set of vertices, $E$ is a
 finite set of {\it directed edges} or {\it arcs}, and $\phi$ is a map from
 $E$ to $V\times V$. If $\phi(e)=(u,v)$, then $e$ is called an edge
 from $u$ to $v$, with {\it initial} vertex $u$ and {\it final} vertex
 $v$. A {\it directed walk} $\Gamma$ in $\vec{G}$ of {\it length} $n$
 from $u$ to $v$ is a sequence $e_1,\ldots, e_n$ of $n$ edges such
 that the final vertex of $e_i$ is the initial vertex of $e_{i+1}$,
 for $1\leq i\leq n-1$. 

  Now let $w:E\rightarrow R$ be a {\it weight function} on $E$ with
  values in some commutative ring $R$. If $\Gamma=e_1,\ldots, e_n$ is
  a walk, then the weight of $\Gamma$ is defined by
  $w(\Gamma)=w(e_1)\cdots w(e_n)$. For $1\leq i,j\leq p$ and $n\in
  \mathbb{N}$, we define
\[
     A_{i,j}(n)= \sum_{\Gamma}w(\Gamma),
\]
 where the sum is over all walks $\Gamma$ in $\vec{G}$ of length $n$
 from $v_i$ to $v_j$. In particular, $ A_{i,j}(0)=\delta_{ij}$. The
 fundamental problem treated by the transfer-matrix method is the
 evaluation of $A_{i,j}(n)$. The idea is to interpret
 $A_{i,j}(n)$ as an entry in  a certain matrix. Define a $p\times p$
 matrix $D=(D_{i,j})$ by
\[
    D_{i,j}=\sum_{e}w(e),
\]
where the sum is over all edges $e$ satisfying that its initial vertex
is $v_i$ and its final vertex is $v_j$. In other words,
$D_{i,j}=A_{i,j}(1)$. The matrix $D$ is called the {\it adjacency
  matrix} of $\vec{G}$, with respect to the weight function $w$.

\begin{theorem}
  Let $n\in \mathbb{N}$. Then the $(i,j)$-entry of $D^n$ is equal to
  $A_{i,j}(n)$. Here we define $D^0=I_p$ even if $D$ is not
  invertible, where $I_p$ is the identity matrix.
\end{theorem}
\begin{proof}
  See \cite{Sta97}.
\end{proof}

%%%%%%%%%%%%%% rectangular lattice again
\subsubsection*{Rectangular lattice again}

 The transfer-matrix method gives us another way to compute $T_{L_{m,n}}(x,y)$ for
 a fixed width $m$ at point ($x$,$y$) which is described in Calkin
 et. al. \cite{CMNN03}. In this case we have the same restriction as before, a fixed
 width $m$ for small values of $m$, but it has the advantage of
 being easily automatized. 

 \begin{theorem}[Calkin et. al. 2003]\label{teo:transfer}
   For indeterminates $x$ and $y$ and integers $n,m \ge 2$, $m$ fixed, we have
 \[
    T_{L_{m,n}}(x+1,y+1) = x^{nm-1} X_m^t \cdot (\Lambda_m)^{n-1}\cdot \vec{1},
 \]
 where $X_m$, a vector of length $c_m$, and $\Lambda_m$, a $c_m\times
 c_m$ matrix, depend on $x$,$y$ and $m$ but not $n$. And $\vec{1}$  is
 the  vector of length   $c_m$ with all entries equal  to~1.
\end{theorem}

 The quantity $c_m$ is  the $m$--th Catalan number, so the method is just
 practical for small values of $m$. Computing the vectors $X_m$ and
 the matrix $\Lambda_m$ can be easily done in a computer. 

 For example, for $L_{2,n}$ we get that $ T_{L_{2,n}}(x+1, y+1)$ equals
 \[
       x^{2n-1}(x^{-1},1)
       \left(\begin{array}{ll}
        x^{-1} + 3x^{-2}+ yx^{-2} & \quad 1+2x^{-1} \\
        x^{-1} + 2x^{-2}+x^{-3}   & \quad 1+2x^{-1}+x^{-2}
        \end{array} \right)^{n-1}
	\left(\begin{array}{l}
        1\\ 1
        \end{array} \right).
 \]
 
%%%%%%%%%%%%%%% Wheels & whrils
\subsubsection*{Wheels and whirls}
 The transfer matrix method takes a nice turn when combined with the physics idea of
 boundary conditions. In this case more lineal algebra is required but the method
 is still suitable to use in a computer algebra package.

 A well-know family of self-dual graphs are wheel graphs, $W_n$. The graph $W_n$ has $n+1$ vertices and $2n$ edges, see Figure~\ref{fig:r-rueda}. The rim of the wheel graph $W_n$  is a circuit-hyperplane of the corresponding graphic matroid, the relaxation of this circuit-hyperplane gives the matroid $W^n$, the whirl matroid on $n$ elements. In~\cite{CS01A}, S.-C. Chang and R. Shrock using results from~\cite{Shr00} compute the Tutte polynomial of $W_n$.
\begin{eqnarray}\label{eq:wheel}
 T_{W_n}(x,y)  &=& \frac{1}{2^n}[(1+x+y)+((1+x+y)^2-4xy)^{1/2}]^n \\ \nonumber
           & & +\frac{1}{2^n}[(1+x+y)-((1+x+y)^2-4xy)^{1/2}]^n+xy-x-y-1.
\end{eqnarray}

\begin{figure}[h!]
\begin{center}
 \includegraphics[scale=0.4]{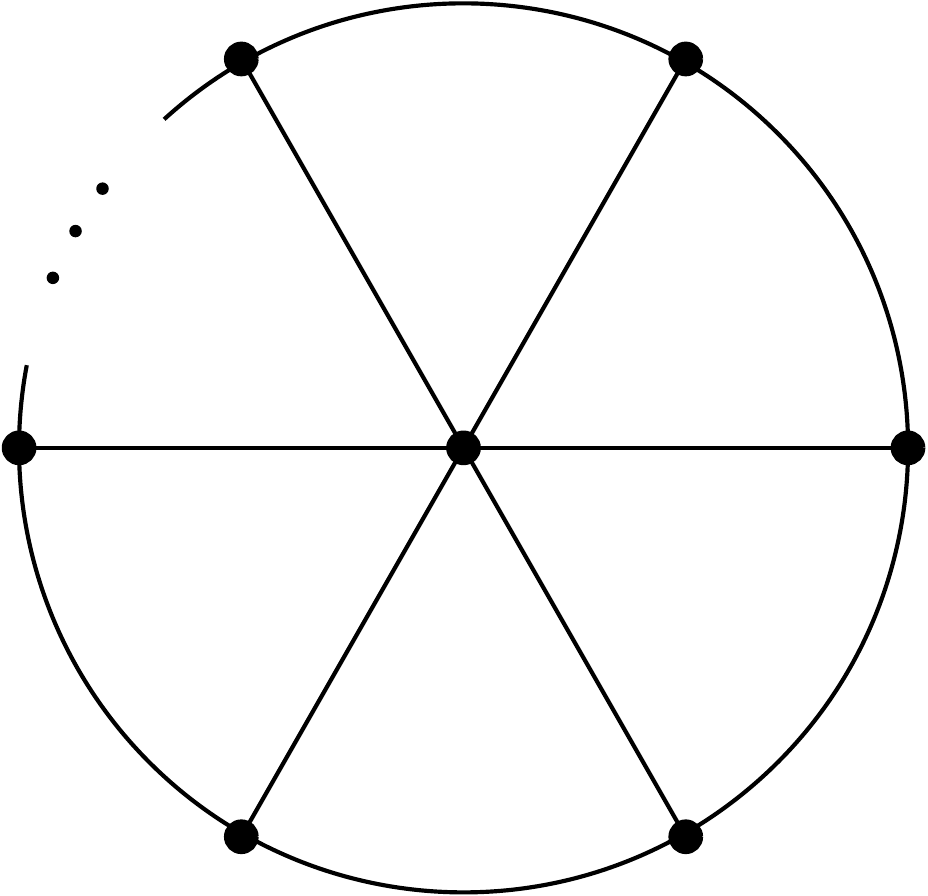}
\caption{The n-wheel}\label{fig:r-rueda}
\end{center}
\end{figure}

From this expression is easy to compute an expression for the Tutte polynomial of whirls using equation~(\ref{eq:relajacion}).
\begin{eqnarray}\label{eq:whirl}
 T_{W^n}(x,y)& = & \frac{1}{2^n}[(1+x+y)+((1+x+y)^2-4xy)^{1/2}]^n\\ \nonumber
             &   & +\frac{1}{2^n}[(1+x+y)-((1+x+y)^2-4xy)^{1/2}]^n - 1.
\end{eqnarray}

 Now, the way S.-C. Chang and R. Shrock  compute  $T_{W_n}(x,y)$ is by using the Potts model partition function
 together with 
 the transfer matrix method. Remarkably, the Tutte polynomial and the Potts model partition function are equivalent. 
 But rather than defining the Potts model we invite the reader to check the surveys  in~\cite{WM00,BE-MPS10}. Here we 
 show how to do the computation using the coboundary polynomial and the transfer method.   

 Let us compute the bad colouring polynomial for $W_n$ when we have 3 colours. For this we define the $3\times 3$ 
 matrix $D_3$. 
 \begin{displaymath}
 \left(
 \begin{array}{ccc}
   t^2& 1 & 1\\
   t & t & 1\\
   t & 1 & t
 \end{array}
 \right)
\end{displaymath}
 The idea is that the entry $ij$ of the matrix $(D_{3})^{n}=D_{3}^n$ will contain all the contributions of bad edges to
 the bad colouring polynomial for all
 the colourings $\sigma$, with $\sigma(h)=1$, $\sigma(1)=i$ and   $\sigma(n)=j$ for  the fan graph $F_{n}$, 
 that is obtained from $W_n$ by deleting one edge from the rim. Then, 
 to get the bad-colouring polynomial of $F_{n}$ we just add all the entries in  $D_{3}^{n}$  and multiply by $\lambda=3$.
 To get the bad-colouring polynomial of $W_n$ we  take the trace of $D_{3}^{n}$ and multiply by $\lambda=3$. This works
 because, by taking the trace you are considering just colourings with the same initial and final configuration, that is you are placing periodic boundary conditions. 

 Thus, for example, the trace in $D_{3}^{3}$ is $t^6+8 t^3 + 6 t^2 + 12 t$, so
  $B_{W_3}(3,t)= 3t^6+24 t^3 + 18 t^2 + 36 t$. By computing the eigenvalues of $D_{3}$ we obtain the bad colouring polynomial of $W_{n}$ with $\lambda=3$. 
\begin{eqnarray*}
 B_{W_n}(3,t) & = & \frac{3}{2^n}[(t^2+t+1)+(t^4-2t^3- t^2+10t+1)^{1/2}]^n \\ \nonumber
         &   & +\frac{3}{2^n}[(t^2+t+1)+(t^4-2t^3- t^2+10t+1)^{1/2}]^n+ 3(t-1)^n.
\end{eqnarray*}

 For arbitrary $\lambda$ we need to compute the eigenvalues of the matrix of $\lambda\times \lambda$ given below 
 \[ \begin{pmatrix}
 t^ 2   & 1      & 1      & \cdots & 1\\
 t      & t      & 1      & \cdots & 1\\
 t      & 1      & t      & \cdots & 1\\
 \vdots & \vdots & \vdots & \ddots & \vdots\\
 t      &  1     & 1      & \cdots & t
 \end{pmatrix}. \]
 
  The eigenvalues of the matrix are 
 \[
   e_{1,2}=\frac{1}{2}[(t^2+t+\lambda-2)\pm(t^4-2t^3-(2\lambda-5)t^2 +(6\lambda-8)t+(\lambda-2)^2)^{1/2}],
 \]
 each with multiplicity 1 and $e_3=(t-1)$ with multiplicity $\lambda-2$. Thus the bad colouring polynomial of 
 $W_n$ is $\lambda(e_{1}^n+ e_{2}^n + (\lambda-2) e_{3}^n)$. To get formula~(\ref{eq:wheel}) we just need to make the
  corresponding change of variables as seen at the beginning of  the previous section.

 Other examples that were computed by using this technique include M\"obius strips, cycle strips and homogeneous clan graphs. For the corresponding 
 formulas you can check~\cite{Shr00,CS00,CS01A,CS01B,CS04}.
 
 %%%%%%%%%%%%%%%%% K-sums
 %%%%%%%%%%%%%%%%%%%%%%%%%%%%%%%%%%%%%%
\subsection{Splitting the problem: 1, 2 and 3-sums}
 Our last technique is based on the recurrent idea of splitting a problem that otherwise 
 may be too big. The notion of connectedness in matroid theory offers a natural
 setting for our purposes, but we do not explain this theory here and 
 we direct the reader to the book of Oxley~\cite{Oxl92} 
 for such a subject. Also, for 1, 2 and 3-sums of binary matroids the book of
  Truemper~\cite{Tru92} is the best reference. 
  
  However, we would like to explain a little of the relation of 1, 2 and 3-sums
  and the splitting  of a matroid. For a matroid $M=(E,r)$, a partition ($X$, $Y$) 
  of $E$ is an \emph{exact} $k$-\emph{separation}, for $k$ a positive integer, if 
  \[ \min\{|X|, |Y|\}\geq k\]
  and
  \[ r(X)+r(Y)-r(M)=k-1.\]  
  Now, a matroid $M$ can be written as a 1-sum
   of two of its proper minors if and only if $M$ has an exact 1-separation;
  and $M$ can be written as a 2-sum
  of two of its proper minors if and only if $M$ has an exact 2-separation. 
  The situation is more complicated in the case of 3-sums and we just want 
  to point out that if a \emph{binary} matroid $M$ has an exact 3-separation 
  ($X$,$Y$),  with $|X|$, $|Y|\geq 4$, then there are binary matroids $M_1$
   and $M_2$ such that $M=M_1\oplus_3 M_2$.

%%%%%%%%%%%%%%%% 1-sum
\subsubsection*{1-sum}
 For $n$ matroids $M_1$, $M_2$, $\ldots$, $M_n$, on disjoint sets $E_1$,
  $E_2$,$\ldots$, $E_n$ the direct sum 
 $M_1\oplus M_2\oplus\cdots\oplus M_n$ is the matroid  $(E,\mathcal{I})$ 
 where $E$ is the union of the ground 
 sets and $\mathcal{I}$ = $\{ I_1\cup I_2 \cup\ldots\cup I_n:
  I_i\in \mathcal{I}(M_i)$ for all $i$ in
  $\{1,2,\ldots,n\}\}$.
  
  Directly from~(\ref{eq:expansion}) it follows that the Tutte polynomial of the 1-sum is given by
  \[
  T_{M_1\oplus M_2\oplus\cdots\oplus M_n}(x,y)=\prod_{i=1}^{n} T_{M_i}(x,y).
  \]
%%%%%%%%%%%%%%%%%%%%%%%%% 2-sum  
\subsubsection*{2-sum}

  Let $M_1=(E_1, \mathcal{I}_1)$ and $M_2=(E_2, \mathcal{I}_2)$  be matroids with $E_1\cap E_2=\{p\}$. 
  If $p$ is not a loop or an isthmus in 
  $M_1$ or $M_2$, then the \emph{2-sum} $M_1\oplus_2 M_2$ of $M_1$ and $M_2$ is the matroid on 
  $E_1\cup E_2\setminus\{p\}$ whose collection of independent sets is $\{I_1\cup I_2: I_1\in\mathcal{I}_1$,
   $I_2\in\mathcal{I}_2$, and either $I_1\cup\{p\}\in\mathcal{I}_1$ or $I_2\cup\{p\}\in\mathcal{I}_2\}$

  In~\cite{AA95}, \cite{Bry71} and \cite{Oxl92}, we find recursive formulas 
  for computing the Tutte polynomial
  of the matroid $M_1 \oplus_2 M_2$, in term of the matroids $M_1$ and 
  $M_2$. Here we present the formula
  given in~\cite{AA95} for $T_{M_1 \oplus_2 M_2}(x,y)$.

\begin{equation}\label{2_sum}
 T_{M_1 \oplus_2 M_2}=  \frac{1}{xy - x - y} 
	\left[T_{M_1/r}\hspace{.3cm} T_{M_1 \setminus r} \right] 
	 \left[
  \begin{array}{c c}
     $x - 1$ & $- 1$ \\
     $- 1$   & $y - 1$ \\
  \end{array}
  \right]
	\left[
	\begin{array}{c}
      T_{M_2/r} \\ T_{M_2 \setminus r}
  	\end{array} \right] ,
\end{equation} 
where here we omit the variables $x, y$ of each Tutte polynomial to avoid a cumbersome notation.

\vspace{.4cm}

As an example, let us take $R_6$ that is the 2-sum of $U_{2,4}$ 
with itself. Observe that the election of the base point is 
irrelevant as its automorphism group is the symmetric group. The matroid $U_{2,4}/r\cong U_{1,3}$
and  $U_{2,4}\setminus r\cong U_{2,3}$. Then, by using~(\ref{2_sum}) we get that $T_{U_{2,4} \oplus_2 U_{2,4}}$
equals
 \[
 \frac{1}{xy - x - y} 
	\left[y^2+y+x\hspace{.3cm}x^2+x+y \right] 
	 \left[
  \begin{array}{c c}
     $x - 1$ & $- 1$ \\
     $- 1$   & $y - 1$ \\
  \end{array}
  \right]
	\left[
	\begin{array}{c}
       y^2+y+x\\ x^2+x+y
  	\end{array} \right].
\] 
 Thus, we obtain that
$T_{R_6}(x,y)=x^3+3x^2+4x+2xy+4y+3y^2+y^3$.

%%%%%%%%%%%%%%%%%%%%%%% 3-sum
\subsubsection*{3-sum}
The best known variant of a $3$-sum of two matroids is called $\Delta$-$sum$, see~\cite{Tru92}. 
For  $3$-conected matroids $M_1$ and $M_2$, the
$\Delta$-sum is usually denoted $M=M_1 \oplus_3^{\Delta} M_2$. When
$M_1$ and $M_2$ are graphic matroids with corresponding graphs being $G_1$ and $G_2$, the graph 
$G=G_1 \oplus_3^{\Delta} G_2$ is obtained by identifying a triangle of $T_1$ of $G_1$ with a triangle $T_2$ 
of $G_2$ into a  triangle $T$, called the connector triangle. Finally,  $G$ is obtained by deleting the edges in
 the connector triangle.
 
Here, we present a formula to compute the Tutte polynomial of the $\Delta$-sum of $M_1$ and $M_2$, in terms of the Tutte polynomial of certain minors of  the original matroids.  The expression for  $T_{M_1 \oplus_3^{\Delta} M_2}(x,y)$ 
was taken from~\cite{AA95} and its proof is based on the concept of bipointed matroids that is an extension of the pointed matroid introduced by Brylawski in~\cite{Bry71}. Also, the work in~\cite{AA95} is more general as the author
 gives an expresion for the Tutte polynomial of a certain type of general parallel connection of two matroids. 

%Se dice que un elemento del conjunto soporte $E$ de un matroide
%$M$, es un $elemento$ $circuito$ si dicho elemento no es un itsmo,
%ni un loop de $M$. Sean $p, q \in E$, se consideran $5$ matroides
%elementales $N_i$, $i=1,2,...,5$, sobre $\{p,q\}$:\\

%\begin{enumerate}
%\item[1.] $N_1$, en el cual $p$ y $q$ son itsmos,
%\item[2.] $N_2$, en el cual $p$ es un itsmo y $q$ es un loop,
%\item[3.] $N_3$, en el cual $p$ es un loop y $q$ es un itsmo,
%\item[4.] $N_4$, en el cual $p$ y $q$ son loops,
%\item[5.] $N_5$, en el cual $p$ y $q$ forman  un circuito.
%\end{enumerate}

Let $M_1$ and $M_2$ be two matroids defined on  $E_1$ and $E_2$, respectively. Let $T$ be equal to 
$E_1 \cap E_2=\{p,s,q\}$, a $3$-circuit and  $N = M_1 |T = M_2 |T$.  We requiere that in $M_1$ there exist circuits 
$U_1 \cup \{s\}$ with $U_1\subseteq E_1 \setminus T$ and $U_2 \cup \{p\}$ with $U_2 \subseteq
E_1 \setminus T$; similarly, we need that in $M_2$ there exist  circuits $V_1 \cup \{s\}$
with $V_1 \subseteq E_2 \setminus T$ and $V_2 \cup \{p\}$ with
$V_2 \subseteq E_2 \setminus T$.

For $i=1,2,...,5$, we consider the following  $5$ minors $Q_i$ of $M_1$, on $E_1 \setminus T$, and  $5$
minors $P_i$ of  $M_2$, on $E_2 \setminus T$.

\noindent $Q_1 = M_1 \setminus p \setminus s \setminus q$, $Q_2 =
M_1 \setminus p /  s \setminus q$, $Q_3 = M_1 / p
\setminus s \setminus q$, $Q_4 = M_1 / p / s / q$ and
$Q_5 = M_1
\setminus p \setminus s / q$.\\

\noindent $P_1 = M_2 \setminus p \setminus s \setminus q$, $P_2 =
M_2 \setminus p / s \setminus q$, $P_3 = M_2 / p \setminus
s \setminus q$, $P_4 = M_2 / p / s / q$ and $P_5 = M_2
\setminus p \setminus s / q$.\\

We take the vectors over $\mathbb{Z}[x,y]$:\\

$q^{\rightarrow} = \left[ T_{Q_1}, T_{Q_2}, T_{Q_3}, T_{Q_4}, T_{Q_5}\right]$,\\

$p^{\rightarrow} = \left[ T_{P_1},T_{P_2},  T_{P_3}, T_{P_4}, T_{P_5} \right]$,\\

\noindent again here we omit the variables $x, y$ of each Tutte polynomial to avoid a cumbersome notation.

Finally, the formula for the 3-sum ($\Delta$-sum) of $M_1$ and $M_2$, it is as follows.

\begin{equation}\label{3_sum}
T_{M_1 \oplus_3^{\Delta} M_2}(x,y)= {(q^{\rightarrow})}^t
\hspace{.1 cm} \left[ \frac{1}{-1-x-y+xy} \hspace{.1 cm} {\bf C}
\right] \hspace{.2 cm} p^{\rightarrow}
\end{equation}
where the matrix , {\bf C} is given by:
\[
\left[
\begin{array}{c c c c c}
\frac{(1-y)^2}{-x-y+xy} & \frac{1-y}{-x-y+xy} &
\frac{1-y}{-x-y+xy} & \frac{2}{-x-y+xy} &
\frac{1-y}{-x-y+xy} \\

\frac{1-y}{-x-y+xy} & 1 & \frac{1}{-x-y+xy} & \frac{1-x}{-x-y+xy}
& \frac{1}{-x-y+xy} \\

\frac{1-y}{-x-y+xy} & \frac{1}{-x-y+xy} & 1 & \frac{1-x}{-x-y+xy}
& \frac{1}{-x-y+xy} \\

\frac{2}{-x-y+xy} & \frac{1-x}{-x-y+xy} & \frac{1-x}{-x-y+xy} &
\frac{(1-x)^2}{-x-y+xy} & \frac{1-x}{-x-y+xy} \\

\frac{1-y}{-x-y+xy} & \frac{1}{-x-y+xy} & \frac{1}{-x-y+xy}
& \frac{1-x}{-x-y+xy} & 1 \\
\end{array}
\right]
\]

 As an example let us take $F_8$ that is the 3-sum of $F_7$ and $F_7^{-}$
 along a 3-circuit. In this case we have for $F_7$ the following table, 
 where in the first column we present the 5 minors we need, then the 
 second column has the corresponding matroid, and the third column has 
 the corresponding Tutte polynomial of that matroid.\\
 
 \begin{center}
\begin{tabular}{|c|c|c|}
\hline \rule[-2ex]{0pt}{5.5ex} Minor & Matroid & polynomial \\ 
\hline \rule[-2ex]{0pt}{5.5ex} $F_7\setminus p \setminus s \setminus q$ &
 $U_{3,4}$ & $x^3+x^2+x+y$ \\ 
\hline \rule[-2ex]{0pt}{5.5ex} $F_7\setminus p /  s \setminus q$ &
 $U_{1,2}\oplus U_{1,2}$ & $(x+y)^2$ \\ 
\hline \rule[-2ex]{0pt}{5.5ex} $F_7/ p\setminus s \setminus q$ & 
 $U_{1,2}\oplus  U_{1,2}$ & $(x+y)^2$\\ 
\hline \rule[-2ex]{0pt}{5.5ex} $F_7/ p / s / q$ & $U_{1,4}$ & 
$y^3+y^2+y+x$ \\ 
\hline \rule[-2ex]{0pt}{5.5ex} $F_7\setminus p \setminus s / q$ & 
$U_{1,2}\oplus U_{1,2}$ & $(x+y)^2$\\ 
\hline 
\end{tabular} 
\end{center}
Similarly for $F_7^{-}$ we have the corresponding table:\\
 \begin{center}
\begin{tabular}{|c|c|c|}
\hline \rule[-2ex]{0pt}{5.5ex} Minor & Matroid & polynomial \\ 
\hline \rule[-2ex]{0pt}{5.5ex} $F_7^{-}\setminus p \setminus s \setminus q$ &
 $U_{3,4}$ & $x^3+x^2+x+y$ \\ 
\hline \rule[-2ex]{0pt}{5.5ex} $F_7^{-}\setminus p /  s \setminus q$ &
  $C_3$ plus a paralell edge & $x^2+x+xy+y+y^2$ \\ 
\hline \rule[-2ex]{0pt}{5.5ex} $F_7^{-}/ p\setminus s \setminus q$ & 
 $U_{1,2}\oplus U_{1,2}$ & $(x+y)^2$\\ 
\hline \rule[-2ex]{0pt}{5.5ex} $F_7^{-}/ p / s / q$ & $U_{1,4}$ & 
$y^3+y^2+y+x$ \\ 
\hline \rule[-2ex]{0pt}{5.5ex} $F_7^{-}\setminus p \setminus s / q$ & 
$U_{1,2}\oplus U_{1,2}$ & $(x+y)^2$\\ 
\hline 
\end{tabular} 
\end{center}

By using~(\ref{3_sum}) with the values of the above tables we get
$T_{F_8}(x,y)=x^4+4x^3+10x^2 +8x +12xy +8y+10y^2+4y^3+y^4$.

There is a general concept of $k-sum$ for matroids and a formula exists for this general notion, however the formula is quit intricate, so we refer the reader to the original paper of Bonin and de Mier in~\cite{BdM03}.

%%%%%%%%%%%%%%%% Thickening
\subsubsection*{Thickening, stretch and tensor product}
 Given a matroid $M$ and a positive integer $k$, the matroid $M^{(k)}$ is the matroid obtained from $M$ by replacing each non-loop element by $k$ parallel elements  and replacing each loop by $k$ loops. The matroid $M^{(k)}$ is called the $k$-\emph{thickening} of $M$. In~\cite{BO92} the following formula is given for the Tutte polynomial of $M^{(k)}$ in terms of that of $M$.
 
\begin{equation}\label{eq:thickening}
 \begin{split}
   &T_{M^{(k)}}(x,y)=\\
   &\quad (y^{k-1}+y^{k-2}+\ldots +y+1)^{r(M)} 
  T_{M}(\frac{y^{k-1}+y^{k-2}+\ldots +y+x}{y^{k-1}+y^{k-2}+\ldots +y+1}, y^k).
 \end{split} 
\end{equation} 

A proof by using the recipe theorem can be read in  the aforementioned reference. Here we hint a simple proof by noticing that any flat of $M^{(k)}$ is the k-thickening of a flat of $M$. Thus, by equation~(\ref{coboundary})
\[  \bar{\chi}{M^{(k)}}(q,t)= \bar{\chi}{M}(q,t^k).\]
And thus, by using equations~(\ref{Chi_to_Tutte}) and~(\ref{Tutte_to_Chi}) we obtained the formula.  

  The dual operation to $k$-thickening is that of $k$-\emph{stretch} that is defined similarly.  The matroid $M_{(k)}$ is the
   matroid obtained  by replacing each non-isthmus of $M$ by $k$  elements in series  and replacing each isthmus by  $k$
    isthmuses. The matroid $M_{(k)}$ is called the $k$-\emph{stretch}. It is not difficult to prove that 
    $M_{(k)}\cong ((M^{*})^{(k)})^{*}$ and so, we obtained the corresponding formula for $T_{M_{(k)}}$. 
  \begin{equation}\label{eq:stretch}
 \begin{split}
   &T_{M_{(k)}} (x,y)=\\
   &\quad (x^{k-1}+x^{k-2}+\ldots +x+1)^{r(M^{*})} 
  T_{M}(x^{k}, \frac{x^{k-1}+x^{k-2}+\ldots +x+y}{x^{k-1}+x^{k-2}+\ldots +x+1}).
 \end{split} 
\end{equation} 

 More generally we have the following operation called the \emph{tensor product}.
 A \emph{pointed} matroid $N_d$ is a matroid on a ground set which includes a distinguished element, the point $d$, 
 which will be assumed to be neither a loop or coloop. For an arbitrary  matroid $M$ and a pointed matroid $N_d$, the 
 \emph{tensor product} $M\otimes N_d$ is the matroid obtained by taking a 2-sum of $M$ with $N_d$ at each point $e$ 
 of $M$ and the distinguished point $d$ of $N_d$. The Tutte polynomial of $M\otimes N_d$, where $M=(E,r)$ is then 
 given by
 \begin{equation}
T_{M\otimes N_d}(x,y)= f^{|E|-r(E)} g^{r(E)} 
                           T_{M}(\frac{(x-1)f+g}{g}, \frac{f+(y-1)g}{f}),
\end{equation}
 where $f=f(x,y)$ and $g=g(x,y)$ are polynomials which are determined by the equations
 \begin{eqnarray*}
     (x-1)f(x,y)+g(x,y)&=& T_{N_d\setminus d}(x,y)\\
     f(x,y)+(y-1)g(x,y)&=& T_{N_d/ d}(x,y).
\end{eqnarray*}

 The proof of the formula uses a generalization of the recipe theorem to pointed matroids and can be found in~\cite{Bry82}, here we follow the exposition in~\cite{BO92,JVW90}. Observe that when a matroid $N$ has a transitive automorphism group the choice of the distinguished point $d$ is immaterial. Thus, if $N$ is  $U_{k,k+1}$, $k\geq 1$, we get the $k$-stretch and if $N$ is  $U_{1,k+1}$, $k\geq 1$, we get the $k$-thickening.

%%%%%%%%%%%%%%%%%%%%%%%% OXLEY'S MATROIDS
%%%%%%%%%%%%%%%%%%%%%%%%%%%%%%%%%%%%%%%%%%%%%%%

\section{Aplication: Small matroids}\label{Sec:small_matroids}
 Let us put these techniques in practice and compute some Tutte
 polynomials for matroids with a small number of elements, these are
 matroids from the appendix in Oxley's book~\cite{Oxl92}. We will try, whenever
 possible, to check the result by using two techniques. Soon, the reader will realize that a fair amount of the matroids considered are sparse paving and that computing the Tutte polynomial for them is quite easy.

%U_2,4
\subsubsection*{Matroid $U_{2,4}$}
 The Tutte polynomial of the uniform matroid $U_{2,4}$ can be computed using 
  equation~(\ref{eq:uniform_1}).
 \begin{equation}
 T_{U_{2,4}}(x,y)=  x^2+2x+2y+y^2.
\end{equation}
 
  Also this matroid is the 2-whirl $W^2$ so you can check this result
   using equation~(\ref{eq:whirl})

\begin{figure}[hbtp]
\begin{center}
 \includegraphics[scale=0.5]{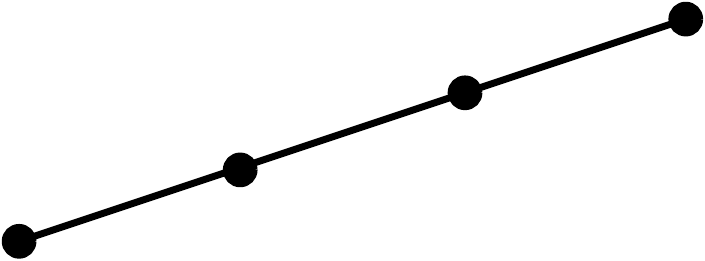}
  \caption{The uniform matroid $U_{2,4}$}
\end{center}
\end{figure}

%K_4 
\subsubsection*{Matroids $W_3$, $W_4$, $W^3$ and $W^4$}

 The 3-wheel $M(W_3)$, which is isomorphic to the graphic matroid $M(K_4)$, 
 is a sparse paving matroid with $\lambda=4$ circuit-hyperplanes,
  so by using~(\ref{eq:sparse_paving}) we get  
\begin{equation}\label{ec:MK4}
 T_{W_3}(x,y)=  x^3+3x^2+2x+4xy+2y+3y^2+y^3.
\end{equation}

 Of course, the above polynomial can be checked using~(\ref{eq:wheel}). The only     relaxation of $M(W_3)$ is the 3-whirl $W^3$, so by using~(\ref{eq:relajacion})
 we obtain its Tutte polynomial. This can be checked by using~(\ref{eq:whirl}).
\begin{equation}\label{ec:W3}
 T_{W^3}(x,y)=  x^3+3x^2+3x+3xy+3y+3y^2+y^3.
\end{equation}
%% W_4 y W^4

 The 4-wheel, $M(W_4)$, and 4-whirl, $W^4$, are matroids whose geometric representation are shown in Figures~\ref{4wheel} and~\ref{4whirl}. Their Tutte polynomial  can be computed using~(\ref{eq:wheel}) and~(\ref{eq:whirl}).
\begin{eqnarray*}
 T_{W_4}(x,y)& = & \frac{(1+x+y+\sqrt{(1+x+y)^2-4xy})^4}{16} + \nonumber \\
             &   & \frac{(1+x+y-\sqrt{(1+x+y)^2-4xy})^4}{16}+xy-x-y-1\nonumber \\
             & = & x^4 +4x^3 +6x^2 +3x +4x^2y +4xy^2 +9xy +3y +6y^2+\nonumber \\
             &   & 4y^3 +y^4.
\end{eqnarray*}

\begin{equation}
 T_{W^4}(x,y) =  x^4 +4x^3 +6x^2 +4x +4x^2y +4xy^2 +8xy +4y+6y^2+
                 4y^3+y^4.
\end{equation}
\begin{figure}[h!] 
\begin{center} 
\begin{minipage}[b]{0.4\linewidth}
\centering
\includegraphics[scale=0.5]{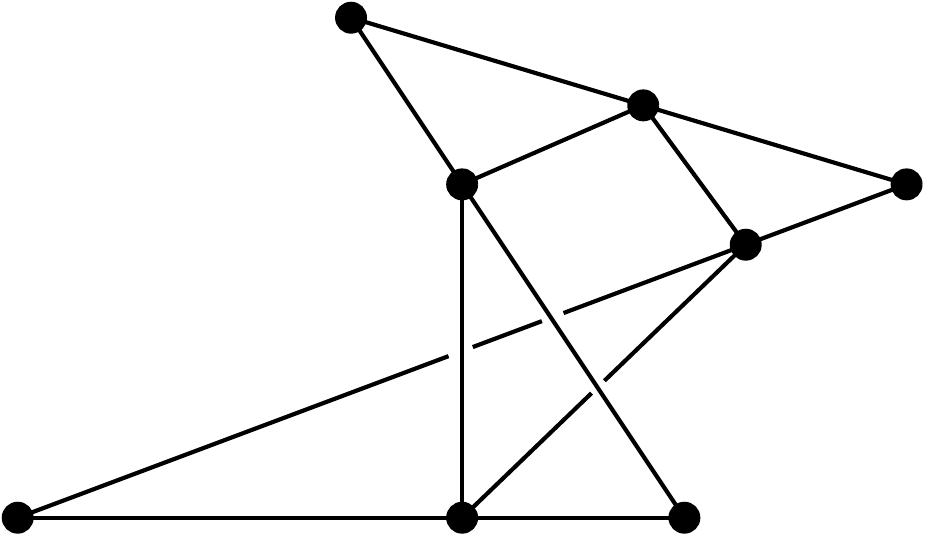}
\caption{The 4-wheel}\label{4wheel}
\end{minipage}
\hspace{0.1cm}
\begin{minipage}[b]{0.4\linewidth}
\centering
\includegraphics[scale=0.5]{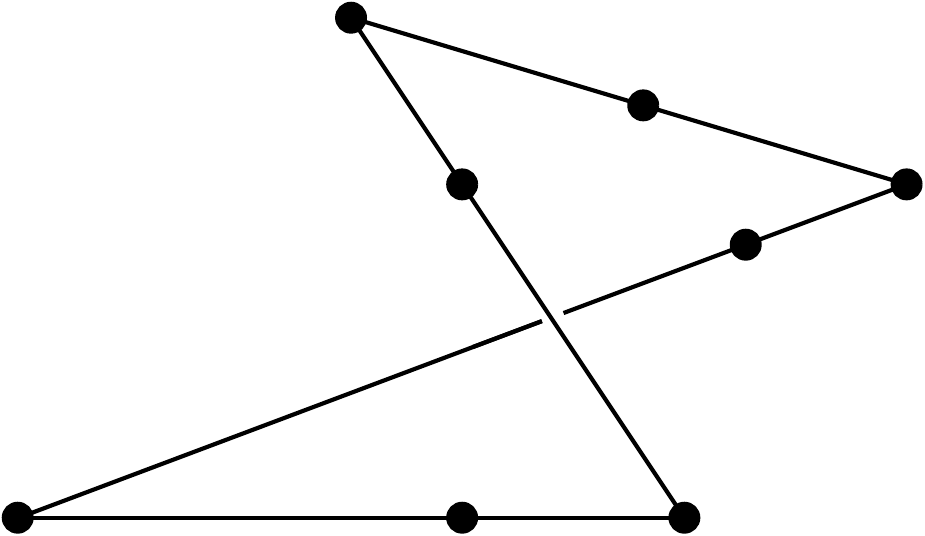}
\caption{The 4-whirl}\label{4whirl}
\end{minipage}
\end{center}
\end{figure}

% Q6, P6, U_3,6
\subsubsection*{Matroids $Q_6$, $P_6$, $R_6$ and  $U_{3,6}$}
 We use (\ref{eq:relajacion}) to compute the Tutte polynomial of $Q_6$, $P_6$ and
  $U_{3,6}$, see Figure~\ref{Fig:1} for a
 geometric representation of these matroids.
 The matroid $U_{3,6}$ is uniform, from (\ref{eq:uniform_1}) we get
\begin{equation*}
 T_{U_{3,6}}(x,y)= x^3 +3x^2+6x+6y+3y^2+y^3.
\end{equation*}
\begin{figure}%[h]
\begin{center}
 \includegraphics[scale=0.55]{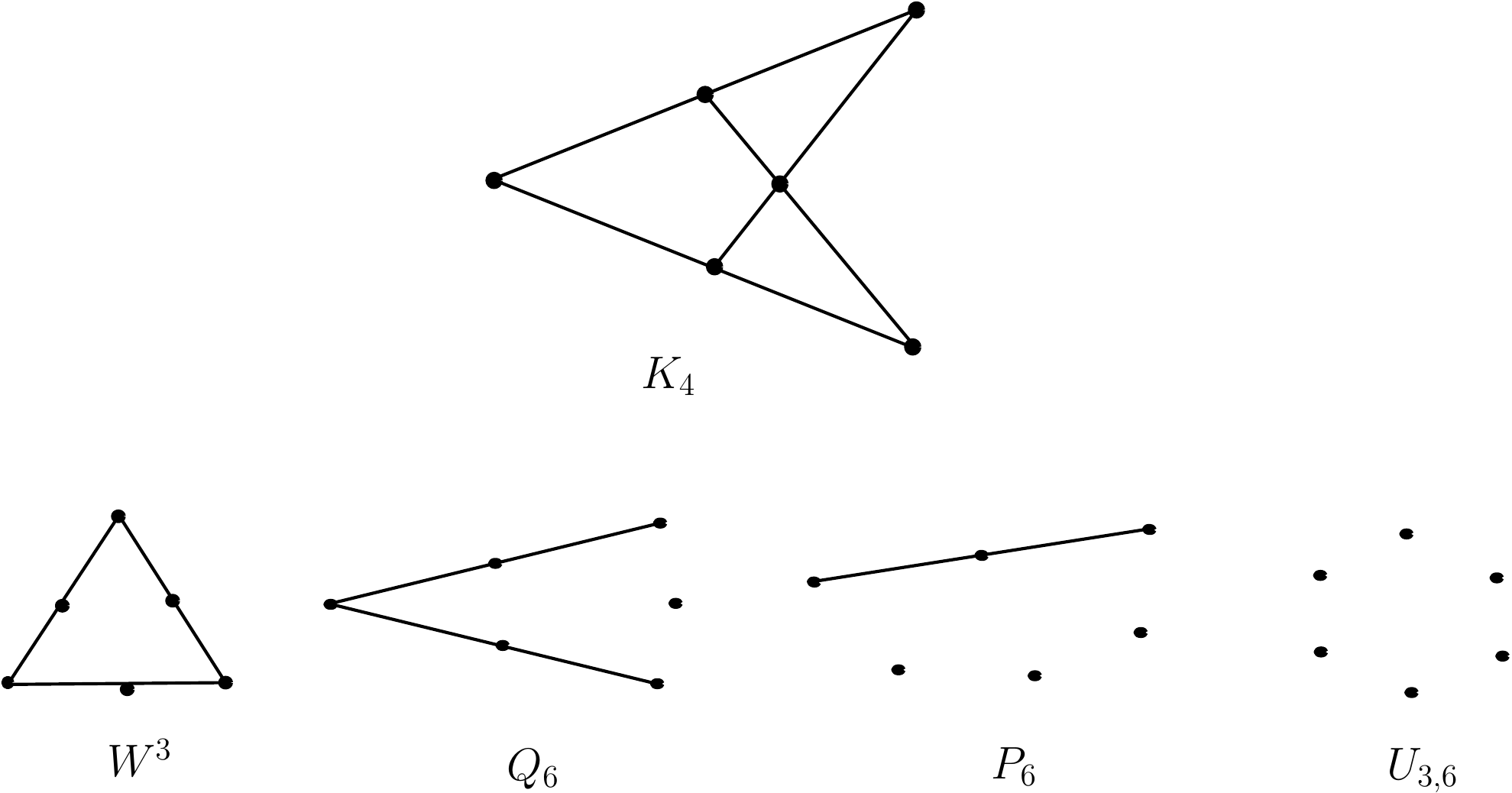}
 \caption{$M(K_4)$, $W^3$, $Q_6$, $P_6$ and $U_{3,6}$}
\label{Fig:1}
\end{center}
\end{figure}
 As $U_{3,6}$ is the only relaxation of $P_6$, 
 from (\ref{eq:relajacion}) its Tutte polynomial is
\begin{equation*}
 T_{P_6}(x,y)=x^3 +3x^2+5x+xy+5y+3y^2+y^3.
\end{equation*}
 Similarly, $P_6$ is the only relaxation of  $Q_6$. Thus,
 from the above equation and (\ref{eq:relajacion}) we obtain
\begin{equation}\label{eq:Q6}
 T_{Q_6}(x,y)=x^3 +3x^2 +4x +2xy +4y +3y^2 +y^3.
\end{equation}

 It is worth mentioning that $P_6$ is a relaxation of $R_6$, see Figure~\ref{fig:R6}. Thus, $Q_6$ and $R_6$ have the same Tutte polynomial.  This is not at all uncommon, see~\cite{Oxl92,Bon03}. Also, $R_6$ is sparse paving so by (\ref{eq:sparse_paving}) its Tutte polynomial is
\begin{eqnarray*}
 T_{R_6}(x,y)&=&(x-1)^3+6(x-1)^2+13x-10+2xy+13y+6(y-1)^2\\
             &  &+(y-1)^3\\
             &=&x^3 +3x^2 +4x +2xy +4y +3y^2 +y^3.
\end{eqnarray*}

\begin{figure}[h!]
\begin{center}
 \includegraphics[scale=0.5]{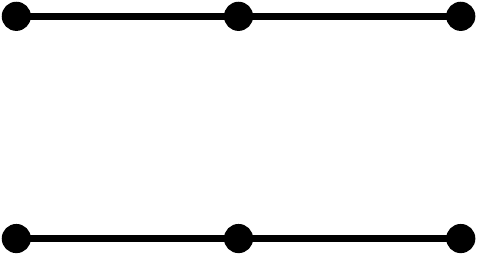}
 \caption{The matroid $R_6$}
 \label{fig:R6}
\end{center}
\end{figure}

% F7 
\subsubsection*{Matroids $F_{7}$, $F_{7}^{-}$ and their duals}

 One of the most mentioned matroid in the literature is the Fano matroid, $F_7$,
 also known as the projective plane $PG(2,2)$ or the unique
 Steiner system $S(2,3,7)$. Being a Steiner triple system implies that the
 Fano matroid is sparse paving, see~\cite{Wel10}. As it has 7 circuit-hyperplanes
 we obtain from~(\ref{eq:sparse_paving}) that the  Tutte polynomial of $F_7$ is
\begin{eqnarray}\label{ec:fano}
 T_{F_7}(x,y) & = & (x-1)^3+7(x-1)^2+14x-21+7xy+28y+21(y-1)^2+ \nonumber \\
              &   & 7(y-1)^3+(y-1)^4 \nonumber\\
              & = & x^3 +4x^2 +3x +7xy  +3y +6y^2 +3y^3 +y^4.
\end{eqnarray}
The matroid $F_7$ is representable over any field of characteristic 2, so the above calculation can be checked  using the program in~\cite{Bar} with the matrix 
\begin{displaymath}
 \left[
 \begin{array}{ccc|rrrr}
  & & & 1 & 1 & 0 & 1 \\
  &I_3 & & 1 & 0 & 1 & 1 \\
  & & & 0 & 1 & 1 & 1 \\
 \end{array}
 \right].
\end{displaymath}
By using~(\ref{eq:duality}) the Tutte polynomial of $F_7^*$ is
\begin{equation}\label{ec:fanodual}
 T_{F_7^*}(x,y)=x^4 +3x^3 +6x^2 +3x +7xy +3y +4y^2 +y^3.
\end{equation}

\begin{figure}[h!]
\begin{center}
\begin{minipage}[b]{0.4\linewidth}
\centering
\includegraphics[scale=0.6]{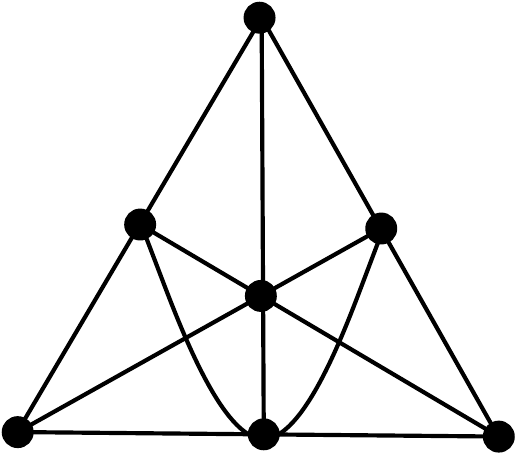}

\end{minipage}
\hspace{0.1cm}
\begin{minipage}[b]{0.4\linewidth}
\centering
\includegraphics[scale=0.5]{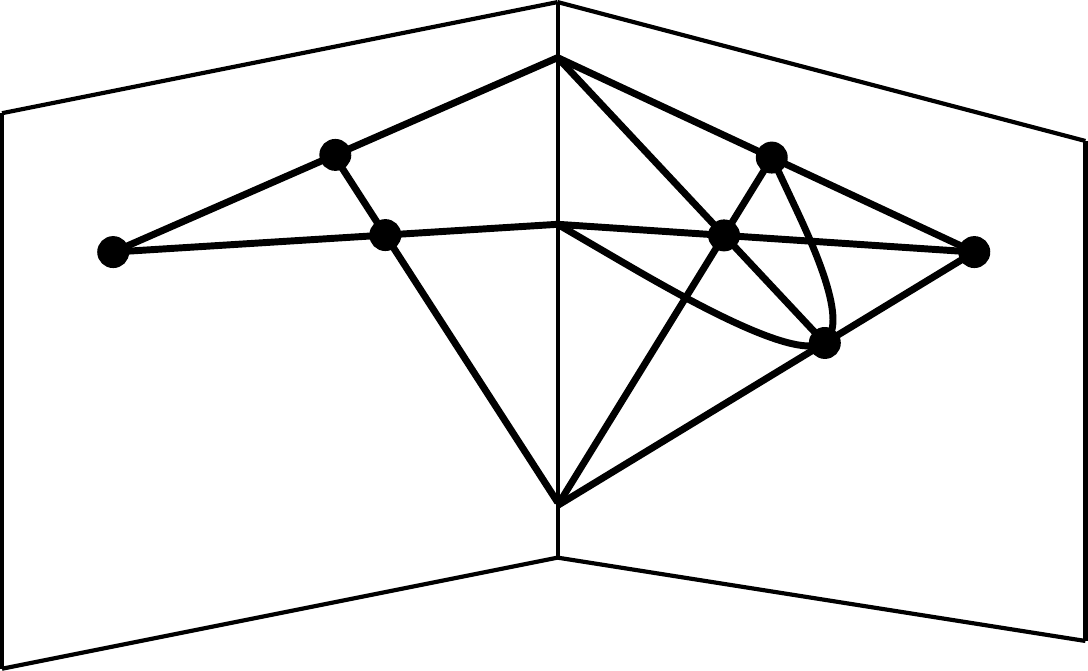}

\end{minipage}
\caption{The Fano matroid and its dual}\label{fig:fano}\label{fig:fanodual}
\end{center}
\end{figure}

The matroids $F_7^-$ and $(F_7^-)^*$ are the corresponding relaxations of  $F_7$ y $F_7^*$, thus we get that
\begin{equation*}
 T_{F_7^-}(x,y)=x^3 +4x^2 +4x +6xy +4y +6y^2 +3y^3 +y^4,
\end{equation*}
and
\begin{equation*}
 T_{(F_7^-)^*}(x,y)=x^4+3x^3+6x^2+4x+6xy+4y+4y^2+y^3.
\end{equation*}

% P_7, P_8 y Q_8
\subsubsection*{Matroids $P_7, P_8$ and $Q_3$}

The matroid $P_7$ is a rank-3 sparse paving matroid, thus by 
using~(\ref{eq:sparse_paving}) we get its Tutte polynomial,
\begin{equation*}
T_{P_7}(x,y)=x^3 +4x^2  +5x +5xy +5y +6y^2 +3y^3 +y^4.
\end{equation*}
 The matrix that represents $P_7$ over a field different from $GF(2)$ is 
\begin{displaymath}
 \left[
 \begin{array}{ccc|rrcr}
  & &    & 1 & 0 & 1 & 1 \\
  &I_3 & & 1 & 1 & 0 & 1 \\
  & &    & a & 1 & a-1 & 0 \\
 \end{array}
 \right]
\end{displaymath}
 with $a\notin\{0,1\}$. Taking $a=2$ we have a representation
 of $P_7$ over $GF(3)$, see~\cite{Oxl92}. Thus, the above calculation of
 the Tutte polynomial can be checked using  the computer program in~\cite{Bar}.

\begin{figure}%[h!]
\begin{center}
\begin{minipage}[b]{0.4\linewidth}
\centering
\includegraphics[scale=0.5]{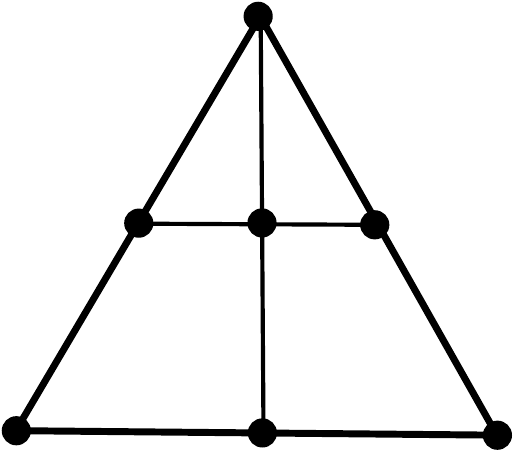}
\label{fig:P7}
\end{minipage}
\hspace{1.4cm}
\begin{minipage}[b]{0.4\linewidth}
\centering
\includegraphics[scale=0.5]{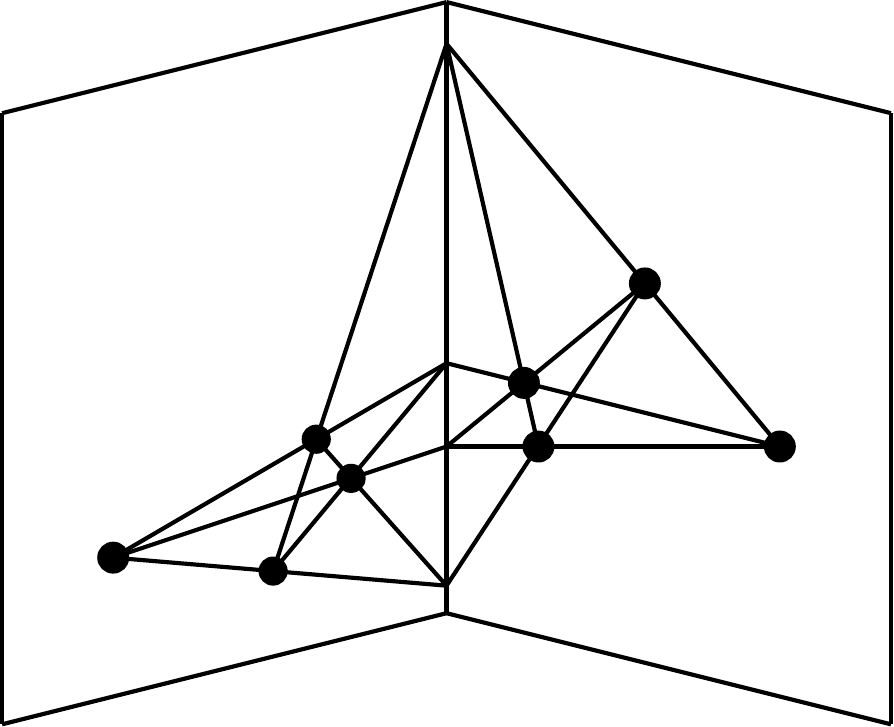}

\end{minipage}
\caption{Geometric representation of $P_7$ and $P_8$}\label{fig:P8}
\end{center}
\end{figure}

 The matroid $P_8$ is also sparse paving and its representation over $GF(3)$ is as below,
\begin{displaymath}
 \left[
 \begin{array}{ccc|rccr}
  & &    & 0 & 1 & 1 & -1\\
  & &    & 1 & 0 & 1 & 1 \\
  &I_4 & & 1 & 1 & 0 & 1 \\
  & &    &-1 & 1 & 1 & 0 \\
 \end{array}
 \right].
\end{displaymath}
 Thus, its Tutte polynomial can be computed using  either~(\ref{eq:sparse_paving})
 or the program in~\cite{Bar}, and you get as a result the following polynomial. 
\begin{equation*}
T_{P_8}(x,y)=x^4 +4x^3 +10x^2  +10x +10xy +10y +10y^2 +4y^3 +y^4.
\end{equation*}

 The rank-3 ternary Dowling geometry  $Q_3$ has geometric representation shown in Figure~\ref{fig:Q3}.
\begin{figure}[h!]
\begin{center}
 \includegraphics[scale=0.4]{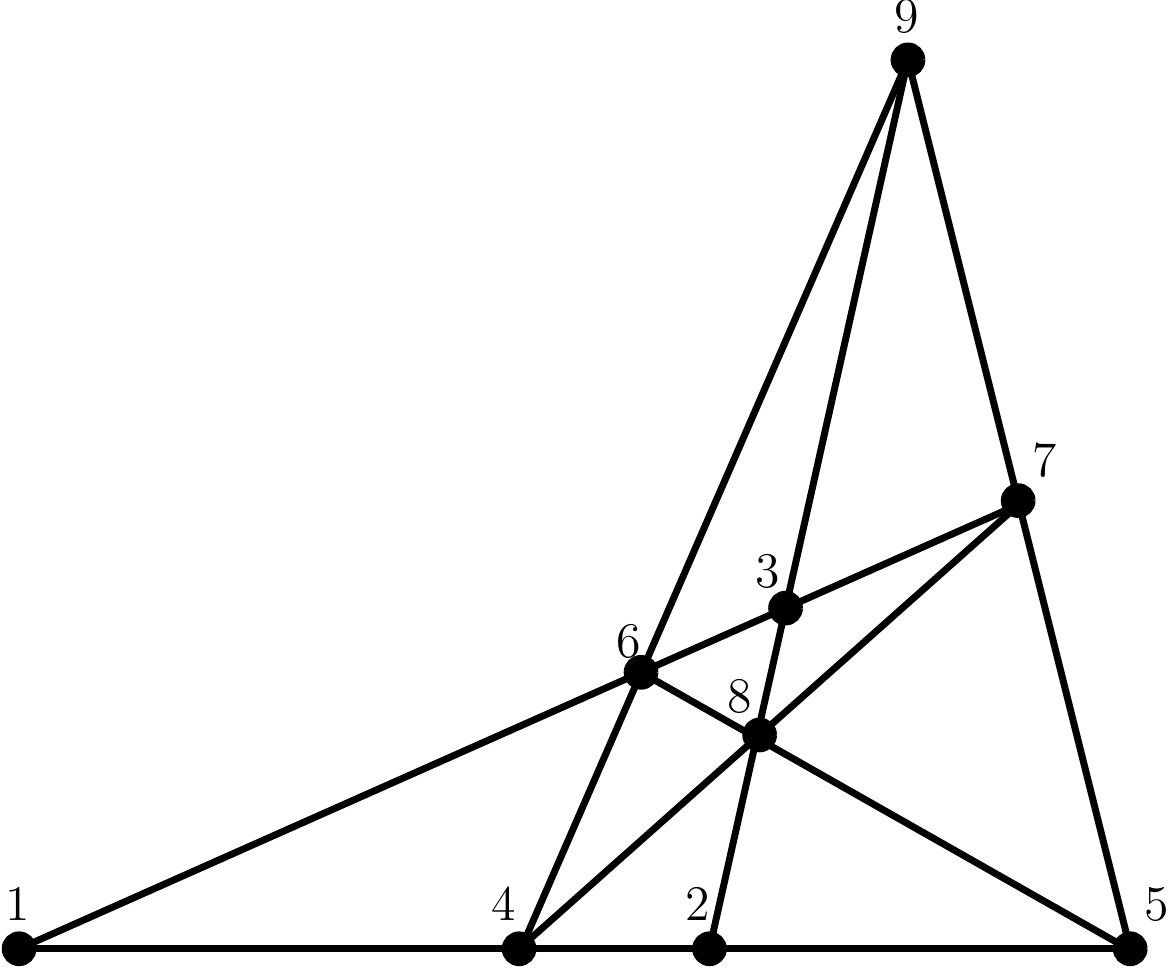}
 \caption{Geometric representation of $Q_3$}
\label{fig:Q3}
\end{center}
\end{figure}

 The matroid $Q_3$ is representable over a field  $F$ if and only if the characteristic of $F$ is different from 2. A matrix that represents  $Q_3$ over $GF(3)$ is
\[
\bordermatrix[{[]}]{ & 1 & 2 & 3 |& 4 & 5 & 6 & 7 & 8 & 9\cr
               & 1 & 0 & 0 |& 1 & 1 & 1 & 1 & 0 & 0 \cr
               & 0 & 1 & 0 |& 1 & -1 & 0 & 0 & 1 & 1\cr
               & 0 & 0 & 1 |& 0 & 0 & 1 & -1 &-1 & 1 \cr}.
\]
 Thus, we can compute its Tutte polynomial using~\cite{Bar} and obtain 
\begin{equation*}
 T_{Q_3}(x,y)= x^3 +6x^2 +8x +3xy^2 +10xy +8y +12y^2 +10y^3 +6y^4 +3y^5 +y^6.
\end{equation*}
Note, however, that this matroid is  paving, so we could have used Proposition~\ref{prop:paving} to get the same result.

%W^3_+
\subsubsection*{Matroid $W^3_+$}

The matroid $W^3_+$ is obtained from $W^3$ by adding an element in parallel and so, the matroid is not paving. In this case, we can use Definition~\ref{def:recursive} to compute the Tutte polynomial. In Figure~\ref{fig:W3cyb} we show how we are using deletion and contraction to find matroids where the Tutte polynomial is either known or easy to compute. 
\begin{eqnarray*}
T_{W^3_+}(x,y) & = & T_{W^3}(x,y)+y\cdot(T_{U_{2,4}}(x,y))+T_{U_{1,3}\oplus U_{0,2}}(x,y)\\
               & = & (x^3+3x^2+3xy+3x+3y+3y^2+y^3)+y\cdot (x^2+2x+2y+y^2)\\
               &   & +(y^4+y^3+xy^2)\\
               & = & x^3+3x^2 +3x +x^2y +5xy+xy^2+3y+5y^2+3y^3+y^4.
\end{eqnarray*}
\begin{figure}%[h!]
\begin{center}
 \includegraphics[scale=0.52]{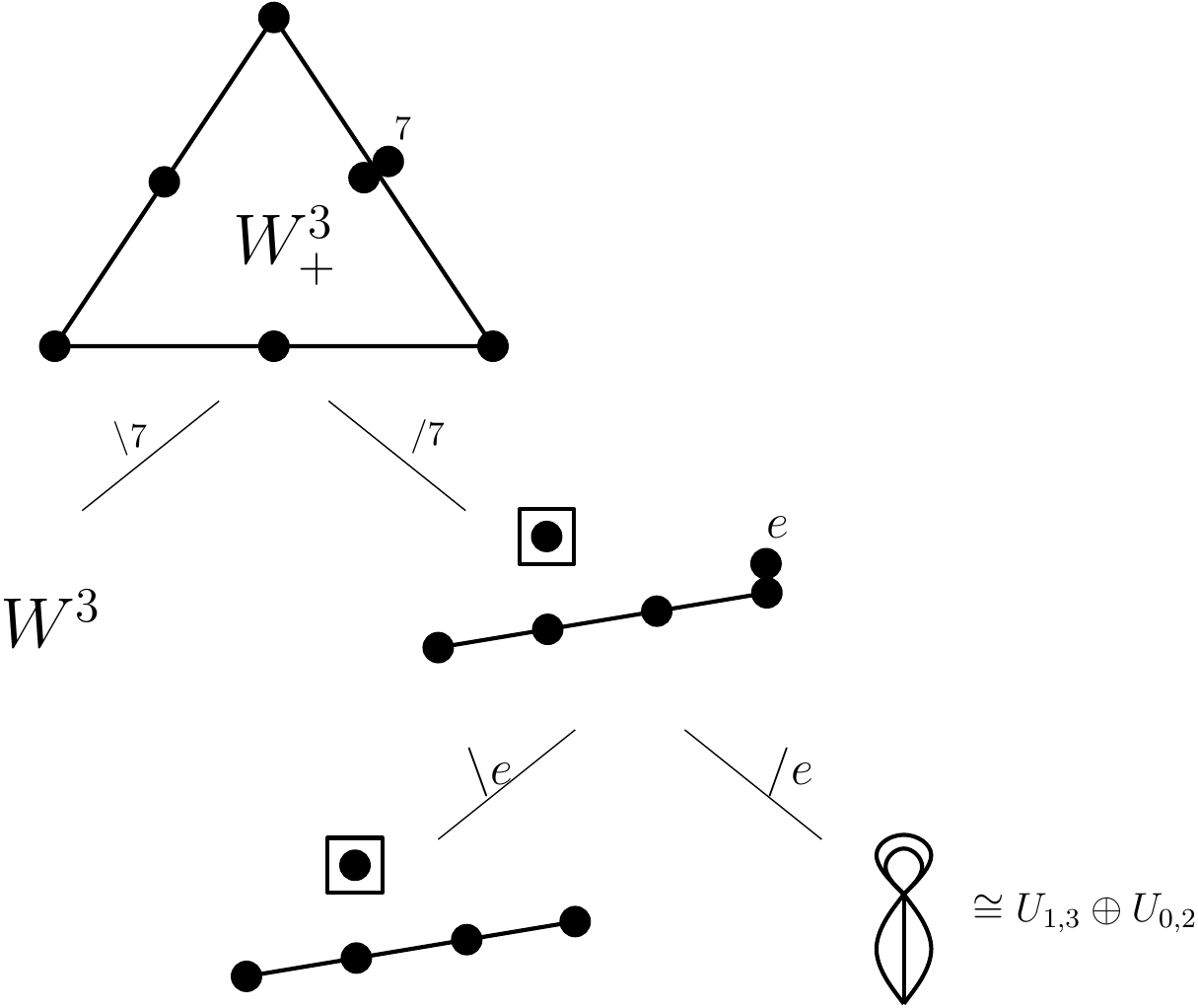}
 \caption{Deletion and contraction reduction  for $W^3_+$}
 \label{fig:W3cyb}
\end{center}
\end{figure}

% AG(3,2), R8 F8 y L8
\subsubsection*{Matroids $AG(3,2)$,  $AG(3,2)'$, $R_8$, $Q_8$, $F_8$ and $L_8$}

 The second affine plane that we consider here is $AG(3,2)$, 
 Figures~\ref{fig:ag32} and~\ref{fig:cubonum1} show two ways
 of representing the matroid. For example, in Figure~\ref{fig:cubonum1} 
\begin{figure}[h!]
\begin{center}
\begin{minipage}[b]{0.4\linewidth}
\centering
\includegraphics[scale=0.6]{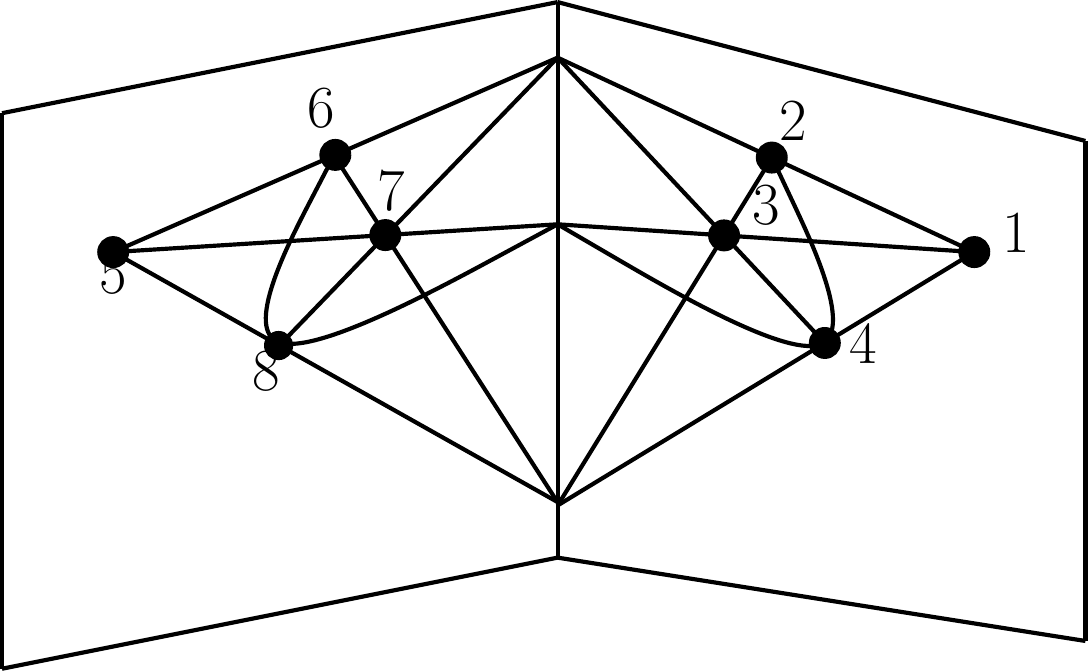}
\caption{The affine plane $AG(3,2)$}
\label{fig:ag32}
\end{minipage}
\hspace{1.4cm}
\begin{minipage}[b]{0.4\linewidth}
\centering
\includegraphics[scale=0.5]{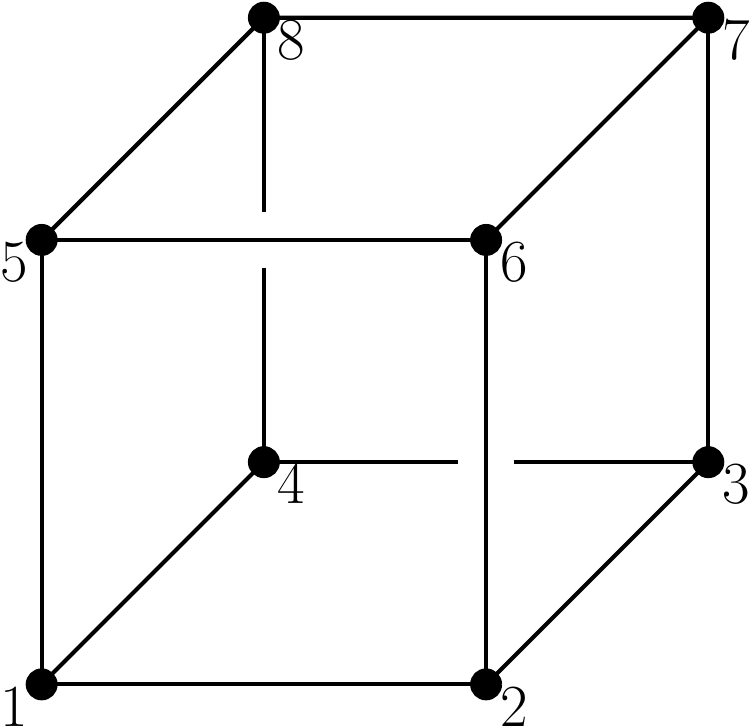}
 \caption{The cube with 8 points}
\label{fig:cubonum1}
\end{minipage}
\end{center}
\end{figure}
 the planes with 4 points are the six faces of the cube, the six diagonal planes
 $\{1,2,7,8\},\{2,3,5,8\},\{3,4,5,6\},$ $\{1,4,6,7\},\{1,3,5,7\}$ and 
 $\{2,4,6,8\}$, plus the two twisted planes $\{1,3,6,8\}$ and $\{2,4,5,7\}$.
 Each of these planes is a circuit-hyperplane. It is not difficult to check
 that $AG(3,2)$ is isomorphic to the unique Steiner system $S(3,4,8)$,
 thus, as explained for $AG(2,3)$, it is sparse paving and its Tutte polynomial is 
\begin{equation}
 T_{AG(3,2)}(x,y)=x^4+4x^3+10x^2+6x +14xy +6y+10y^2+4y^3+y^4.
\end{equation}

 By using  (\ref{eq:relajacion}), we can compute the Tutte polynomial of $AG(3,2)'$,
 the unique relaxation of $AG(3,2)$ that here we obtained by relaxing the twisted
 plane $\{2,4,5,7\}$. Thus, the resulting polynomial is
{\setlength\arraycolsep{1pt}
\begin{equation}\label{eq:ag32}
 T_{AG(3,2)'}(x,y) = x^4 +4x^3+10x^2+7x +13xy +7y+10y^2+4y^3+y^4.
\end{equation}}

 Now, $AG(3,2)'$ has two relaxations, $R_8$ and $F_8$. The matroid $R_8$ is obtained
 by relaxing from $AG(3,2)'$ the 
 other twisted plane $\{1,3,6,8\}$.  On the other hand, $F_8$, is obtained from 
 $AG(3,2)'$ by relaxing a diagonal plane. The geometric representation of $F_8$ is shown in Figure~\ref{fig:F8}. 
\begin{figure}[h!]
\begin{center}
  \includegraphics[scale=0.5]{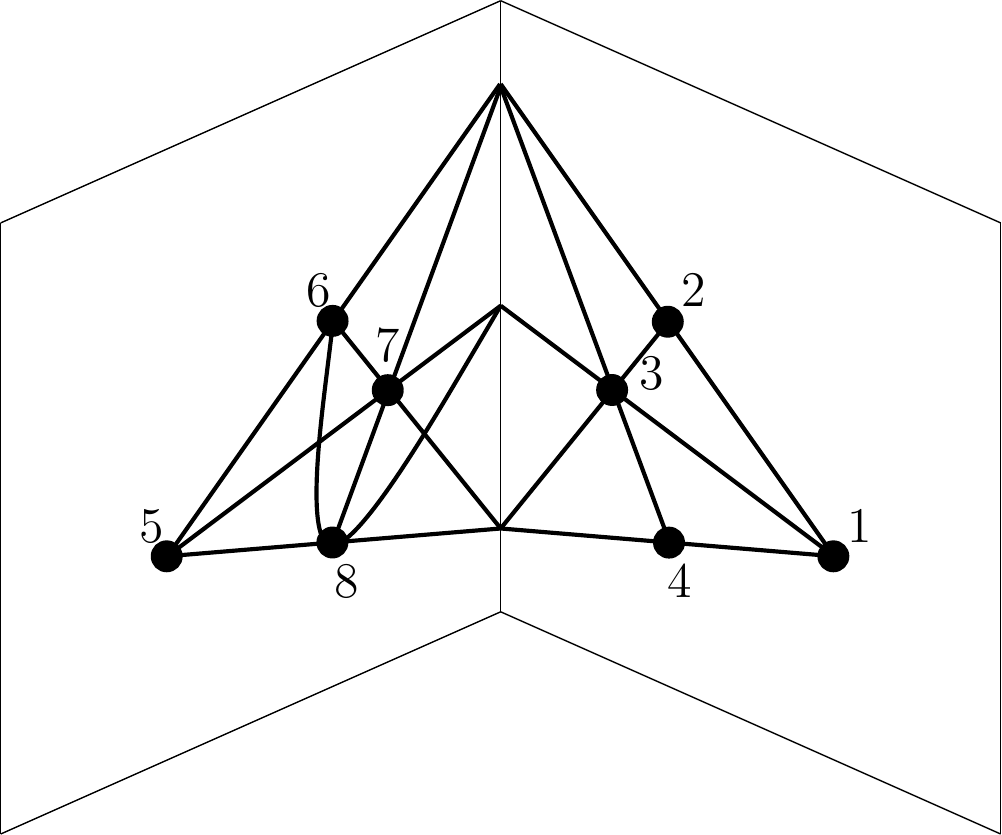}
  \caption{The matroid $F_8$}
  \label{fig:F8}
\end{center}
\end{figure}

 The matroid $R_8$ is representable over any field except for $GF(2)$ while
 $F_8$ is not representable; however, they have the same Tutte polynomial as both are relaxations of the same matroid. 
\begin{equation*}
 T_{R_8}(x,y)=T_{F_8}(x,y)=x^4+4x^3+10x^2 +8x +12xy +8y+10y^2+4y^3+y^4.
\end{equation*}

 The unique relaxation of $R_8$ is $Q_8$, and is obtained by relaxing one of
 the six diagonal planes of $R_8$. From the previous equation
 and~(\ref{eq:relajacion}) we get
\begin{equation*}
 T_{Q_8}(x,y)=x^4+4x^3+10x^2 +7x +11xy +7y+10y^2+4y^3+y^4.
\end{equation*}
 To end this subsection we consider the matroid $L_8$ which is a rank-4
  sparse paving matroid with 8 elements and its
 circuit-hyperplanes are the six faces of the cube plus the two twisted
  planes $\{1,8,3,6\}$ and $\{2,7,4,5\}$, 
 see Figure~\ref{fig:cubonum1}. Thus, its Tutte polynomial is
\begin{eqnarray*}
 T_{L_8}(x,y)& = & (x-1)^4+8(x-1)^3+28(x-1)^2+48x-42+8xy+48y+\\
             & & 28(y-1)^2+8(y-1)^3+(y-1)^4\\
             & = & x^4+4x^3+10x^2 +12x +8xy +12y+10y^2+4y^3+y^4.
\end{eqnarray*}

% S_8, T_8 y J
\subsubsection*{Matroids $S_8,T_8$ and $J$}

 The matroid $S_8$ has a geometric representation shown in  Figure~\ref{fig:S8}
 together with its representation over $GF(2)$. Its Tutte polynomial can be
 computed using the program in~\cite{Bar}. 
\begin{equation}\label{eq:tutte_s8}
 T_{S_8}(x,y)= x^4+ 4x^3+ 7x^2 + 4x+ 10x y + 3x y^2 +
              3x^2y+ 4y + 7y^2 + 4y^3 + y^4.
\end{equation}

 Note that the matroid is self-dual but is not paving as it has a 3-circuit. However, we can check the above computation using deletion and contraction. When we contract the forth column in the representation we obtain the representation of $F_7$, and we have already computed the Tutte polynomial of this matroid. Now, when we delete the same element,  we obtain a rank 4 graphic matroid. The graph is $K_{2,4}$ with an edge contracted. If we call this graph $H$, by using~(\ref{eq:deletion-contraction}) we obtain $T_{H}=T_{K_{2,4}}-y\,T_{K_{2,3}}$. These polynomial can be computed either using the general method for complete bipartite graphs given in~(\ref{Formula_Knm}) or by using the formula for computing the Tutte polynomial of the 2-stretching of the graphs with two and three parallel edges respectively, given in~(\ref{eq:stretch}). 
 In both cases, we get
 \begin{equation}\label{eq:tutte_H}
  T_{H}= x^4+ 3x^3+ 3x^2+x+3xy+3x^2y+3xy^2+y+y^2 + y^3.
  \end{equation}
  By adding~(\ref{eq:tutte_H}) and~(\ref{ec:fano}) we get~(\ref{eq:tutte_s8})
 %We can obtain an standard representation by Gaussian reducing on the last column. 
 %\[
  %\bordermatrix[{[]}]{ & 1 & 2 & 3 & 4 & 5 & 6 & 7 \cr
   %            & 0 & 1 & 0 |& 1 & 1 & 1 &  1 & 0 & 0 \cr
   %           & 1 & 0 & 0 |& 1 & -1 & 0 & 0 & 1 & 0\cr
   %            & 1 & 1 & 1 |& 0 & 0 & 1 &  0 & 0 & 1 \cr}.
  %\]

\begin{figure}%[h!]
\begin{center}
\begin{minipage}[b]{0.4\linewidth}
\includegraphics[scale=0.5]{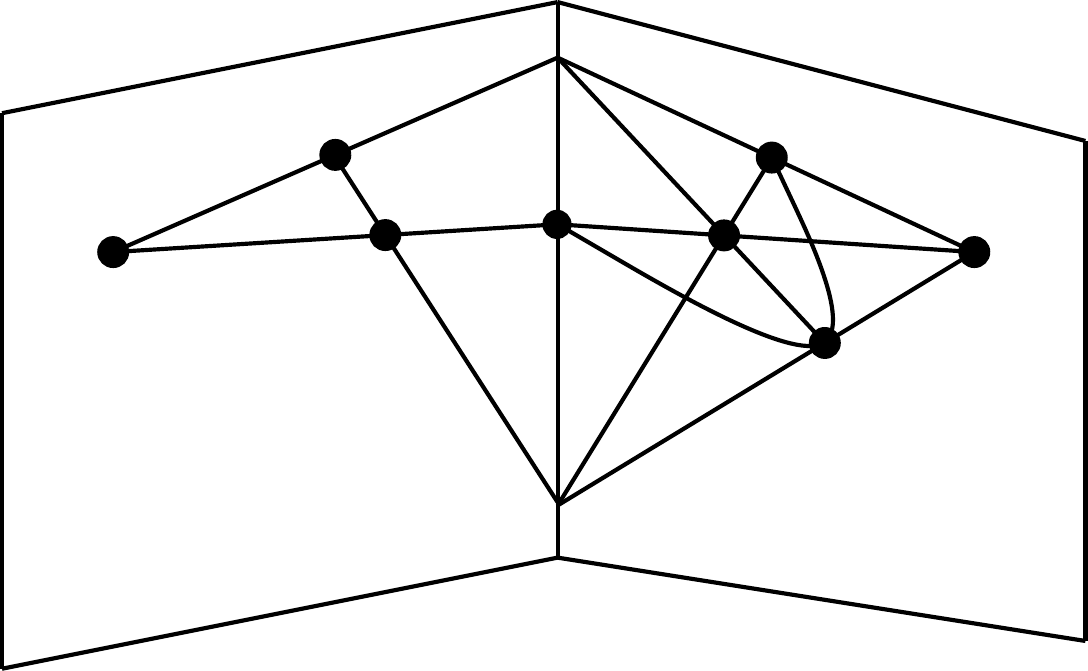}
\end{minipage}
\hspace{1.4cm}
\begin{minipage}[b]{0.4\linewidth}
\includegraphics[scale=0.8]{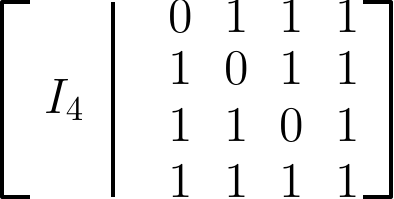}
\end{minipage}

\caption{On the left hand side we give the geometric representation of
         the matroid $S8$ and on the right hand side the matrix that
         represents it over $GF(2)$.}\label{fig:S8}
\end{center}
\end{figure}

 The matroid  $T_8$ is representable over a field $\mathbb{F}$ if and only if
 the characteristic is 3. We show its geometric
 representation in Figure~\ref{fig:T8}. A representation of $T_8$ over $GF(3)$ is $[I_4 | J_4-I_4]$, where $J_4$ is the matrix of 1's. This matroid is
 self-dual and sparse paving so its Tutte polynomial is
\begin{equation*}
 T_{T_8}(x,y)=x^4+4x^3+10x^2+9x +11xy +9y+10y^2+4y^3+y^4.
\end{equation*}
\begin{figure}[h!]
\begin{center}
\includegraphics[scale=0.5]{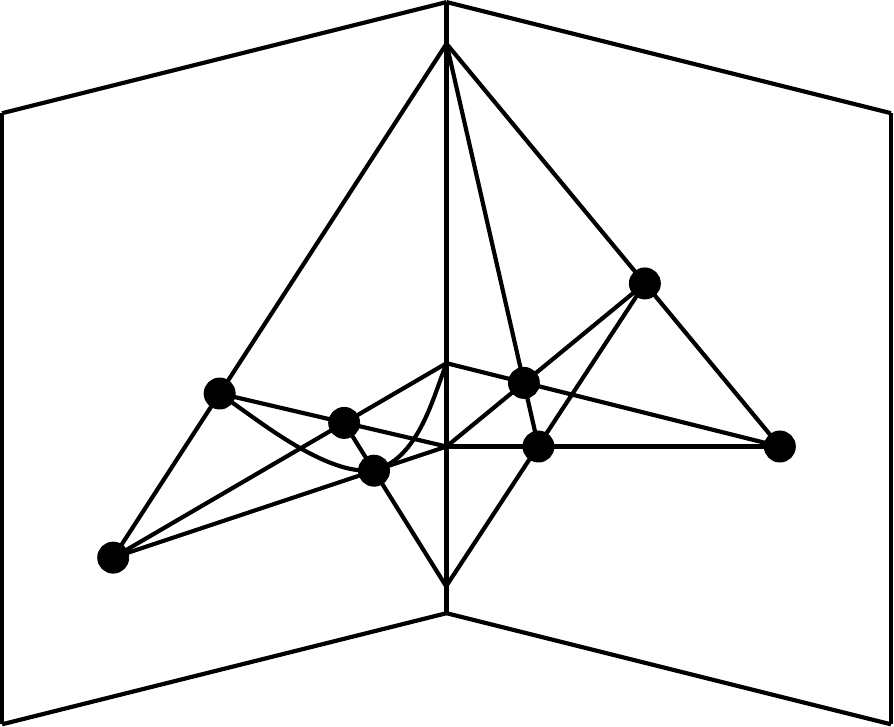}
\caption{The matroid $T8$}\label{fig:T8}
\end{center}
\end{figure}

 In Figure~\ref{fig:J} we show the geometric representation of the matroid $J$. It is a self-dual matroid that is not paving, as it has a 3-circuit, and its representation over $GF(3)$ is
\begin{displaymath}
 \left[
 \begin{array}{ccc|rrcr}
  & &    & 1 & 0 & 0 & 1\\
  & &    & 1 & 1 & 1 & 0 \\
  &I_4 & & 0 & 1 & 0 & 1 \\
  & &    & 0 & 0 & 1 & 1 \\
 \end{array}
 \right]
\end{displaymath}
 where the labelling of the columns correspond to the labelling of the elements in the geometric representation 
 of Figure~\ref{fig:J}. The Tutte polynomial can be computed using~\cite{Bar}.
\begin{equation*}
 T_J(x,y)=x^4 +4x^3 +7x^2 +6x +3x^2y  +3xy^2 +8xy +6y +7y^2 +4y^3 +y^4.
\end{equation*}

\begin{figure}[h!]
\begin{center}
\includegraphics[scale=0.45]{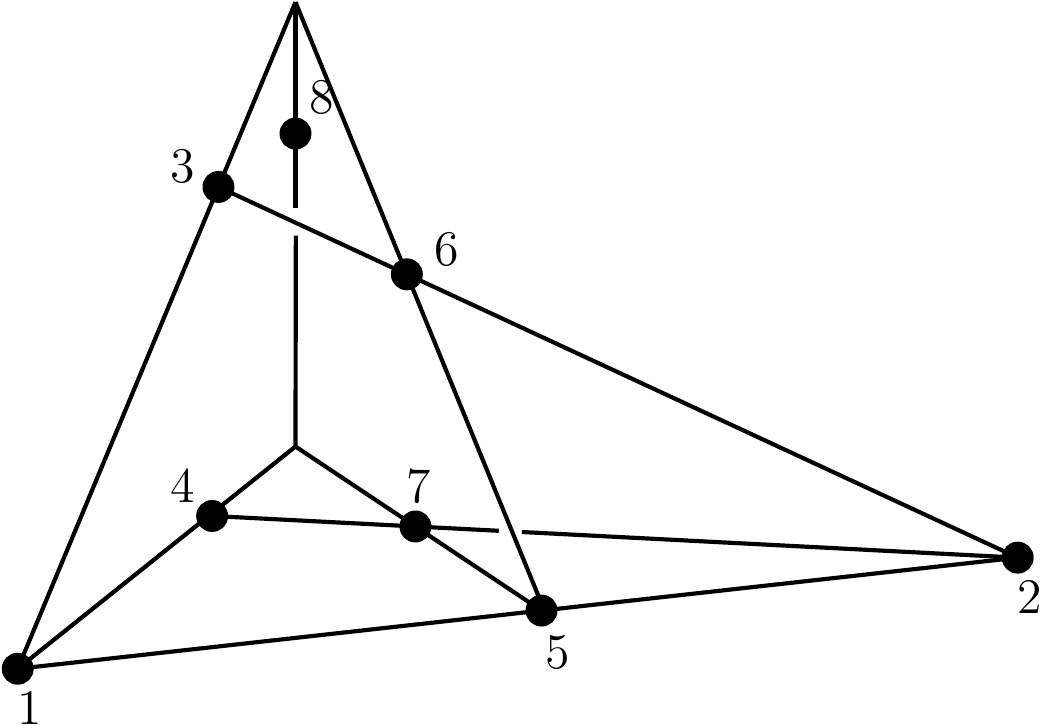}
 \caption{The matroid $J$}
\label{fig:J}
\end{center}
\end{figure}

 This computation can be checked by, for example, contracting the element labelled 2 to 
 obtain the graphic matroid of the graph that is the 2-stretching of $C_4$ minus an edge.
 Now, the matroid $J\setminus 2$ is sparse paving and contains 5 circuit-hyperplanes.
 By adding the 
 Tutte polynomial of these two matroids you get the same result as above. 

% Vamos matroid
\subsubsection*{The V\'amos matroid}
 The 8-element rank-4 matroid whose geometric representation is shown in Figure~\ref{fig:V8} is known as the \emph{V\'amos} matroid and it is usually denoted by $V_8$. It occurs quite frequently in Oxley's book and has many interesting properties, for example is not representable over any field and it is sparse paving, see~\cite{Oxl92}.  A related matroid $V_8^{+}$ has the same ground set and is defined in the same way as $V_8$ but with $\{5,6,7,8\}$ added as a hyperplane. In fact, $V_8$ is obtained from $V_8^{+}$ by relaxing $\{5,6,7,8\}$. Notice that both are self-dual. Thus the Tutte polynomial of these two matroids are
\begin{eqnarray}
T_{V_8}(x,y) & = & (x-1)^4+8(x-1)^3+{8\choose 2}(x-1)^2+{8\choose 3}(x-1) \nonumber\\
 & &+{8\choose 4}+5(xy-x-y) +{8\choose 5}(y-1)+{8\choose 6}(y-1)^2 \nonumber \\
& & +8(y-1)^3+(y-1)^4 \nonumber \\
& = & x^4+4x^3+10x^2+15x +5xy +15y+10y^2+4y^3+y^4
\end{eqnarray}
and
\begin{equation}
T_{V_8^{+}}(x,y) = x^4+4x^3+10x^2+14x +6xy +14y+10y^2+4y^3+y^4.
\end{equation}
\begin{figure}[h!]
 \begin{center}
    \includegraphics[scale=0.7]{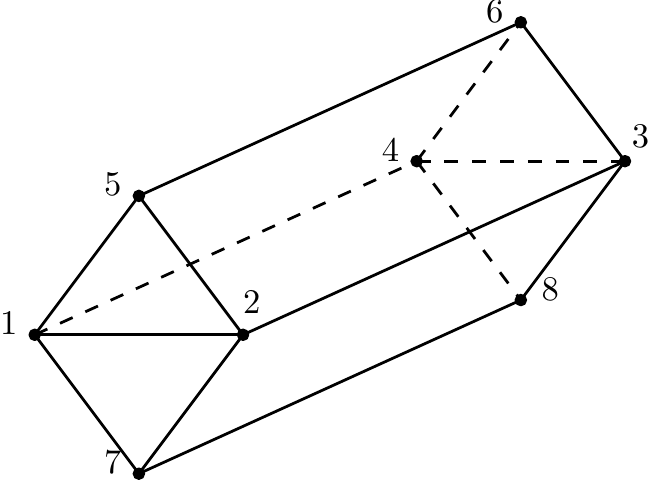}
  \caption{Geometric representation of $V_8$}
   \label{fig:V8}
   \end{center}
\end{figure}

%R_8, R_9, R_10 and R_12
\subsubsection*{Matroids $R_8,R_9,R_{10}$ and $R_{12}$}
 The real affine cube, $R_8$, is represented over  all fields of characteristic other than two by the matrix
\begin{displaymath}
 \left[
 \begin{array}{ccc|rrrr}
  & &    & -1 & 1 & 1 & 1  \\
  &I_4 & & 1 & -1 & 1 & 1  \\
  & &    & 1 & 1 & -1 & 1 \\
  & &    & 1 & 1 & 1 & -1 \\
   \end{array}
 \right].
\end{displaymath}
 The matroid is sparse paving and its Tutte polynomial is 
\begin{equation}
 T_{R_8}(x,y)=x^4+4x^3+10x^2+8x +12xy +8y+10y^2+4y^3+y^4.
\end{equation}

 On the contrary, the ternary Reid geometry, $R_9$, is not paving but it is representable if and only if the characteristic of the field is three. The matrix that represents $R_9$ over  $GF(3)$ is
\begin{displaymath}
 \left[
 \begin{array}{ccc|rrrrrr}
  & &    & 1 &  1 & 1 & 1 & 1 & 1\\
  &I_3 & & 1 & -1 & -1 & -1 & 1 & 0 \\
  & &    & 0 &  0 & 1 & -1 & 1 & -1 \\
 \end{array}
 \right]
\end{displaymath}
 and its Tutte polynomial can be computed using the program in~\cite{Bar}.
\begin{equation*}
 T_{R_9}(x,y)=x^3 +6x^2  +8x+ 11x y + 2 x y^2+ 8 y+13y^2+10y^3+6y^4+3y^5+y^6.
\end{equation*}
 This was computed before in section~\ref{subsec:internal_activity} with
 the same result. 

 The unique 10-element regular matroid that is neither graphic or cographic, $R_{10}$, has the property that any every single-element deletion is isomorphic to $M(K_{3,3})$, and every single-element contraction is isomorphic to $M^{*}(K_{3,3})$, then by using the polynomial in~(\ref{tute_K33})  and~(\ref{eq:duality}) we obtain
\begin{eqnarray}\label{eq:R10}
 T_{R_{10}}(x,y) & = & x^5 +5x^4 +15x^3 +20x^2 +10x +15x^2y +30xy +15xy^2  +\nonumber \\
                 & & 10y+20y^2+15y^3+5y^4+y^5.
\end{eqnarray}

 Another important regular matroid that is neither graphic or cographic is $R_{12}$,
  which has a matrix  representation over $GF(2)$ given by
\begin{displaymath}
 \left[
 \begin{array}{ccc|rrrrrr}
  & & & 1 & 1 & 1 & 0 & 0 & 0 \\
  & & & 1 & 1 & 0 & 1 & 0 & 0 \\
  &I_6 & & 1 & 0 & 0 & 0 & 1 & 0 \\
  & & & 0 & 1 & 0 & 0 & 0 & 1 \\
  & & & 0 & 0 & 1 & 0 & 1 & 1 \\
  & & & 0 & 0 & 0 & 1 & 1 & 1 \\
  \end{array}
 \right].
\end{displaymath}
 As this matroid is not paving we use the program in~\cite{Bar} to compute its 
 Tutte polynomial. 
\begin{eqnarray*}
 T_{R_{12}}(x,y) & = & x^6 +6x^5 +19x^4 +35x^3 +35x^2 +14x +2x^4y +19x^3y +\\
                 & & 53x^2y +17x^2y^2 +56xy +53xy^2 +19xy^3 +2xy^4 +14y+\\
                 & & 35y^2+ 35y^3+19y^4+6y^5+y^6.\\
\end{eqnarray*}
%The above calculation can be checked by, for example, deleting and contracting
%the element corresponding to the column-vector $e_1$, to obtain two graphic matroids.
%To compute the Tutte polynomial of those you do again deletion-contraction until you
% find the result or you use another
% computer program to compute the Tutte polynomial for these two graphs.

% Pappus matroid
\subsubsection*{Pappus and non-Pappus matroids}
The geometric representation of the \emph{Pappus} matroid is shown in Figure~\ref{fig:pappus}. From the picture it is clear that any two points are in a unique line and that each line is a circuit-hyperplane in a rank-3 matroid. We conclude that the matroid is sparse paving with  $\lambda=9$ circuit-hyperplanes and its Tutte polynomial is
\begin{eqnarray*}
 T_M(x,y)& = & (x-1)^3+9(x-1)^2+27x-78+9xy+117y+126(y-1)^2+\\
         &  &   84(y-1)^3+36(y-1)^4+9(y-1)^5+(y-1)^6\\
&=&x^3+6x^2+12x+9xy+12y+15y^2+10y^3+6y^4+3y^5+y^6.
\end{eqnarray*}

 The \textit{non-Pappus} matroid is a relaxation of the Pappus matroid and its geometric representation is shown in 
 Figure~\ref{fig:npappus}. From the previous equation and~(\ref{eq:relajacion}) we obtained
\begin{eqnarray*}
 T_M(x,y)&=&x^3+6x^2+12x+9xy+12y+15y^2+10y^3+6y^4+3y^5+y^6\\
         & & -xy+x+y\\
         &=&x^3+6x^2+13x+8xy+13y+15y^2+10y^3+6y^4+3y^5+y^6. 
\end{eqnarray*}
\begin{figure}[h!]
\begin{center}
\begin{minipage}[b]{0.4\linewidth}
\centering
\includegraphics[scale=0.4]{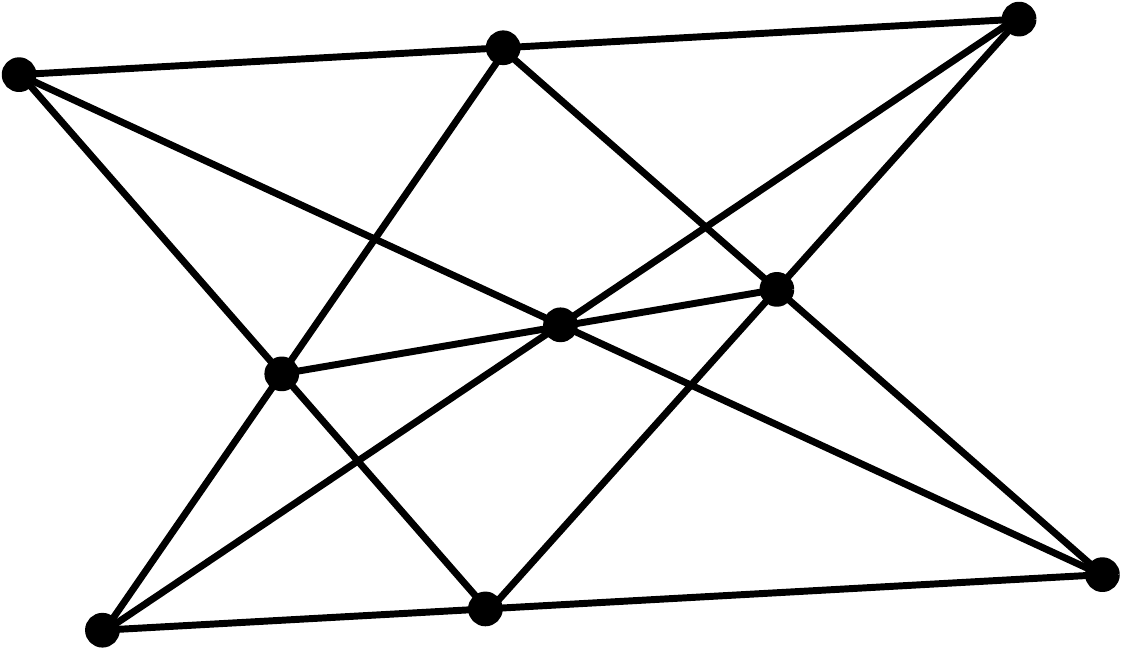}
\caption{Geometric representation of the Pappus matroid}
\label{fig:pappus}
\end{minipage}
\hspace{0.7cm}
\begin{minipage}[b]{0.4\linewidth}
\centering
\includegraphics[scale=0.4]{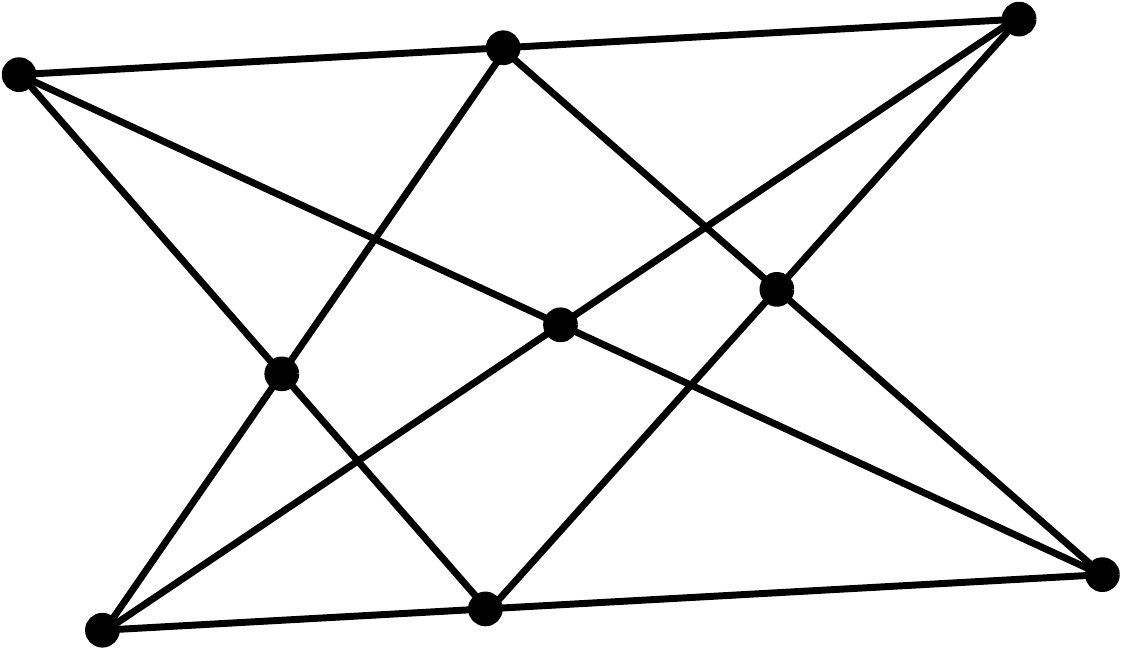}
 \caption{Geometric representation of the non-Pappus matroid}
\label{fig:npappus}
\end{minipage}
\end{center}
\end{figure}

% Non-desargues
\subsubsection*{Matroid non-Desargues}
 The \emph{non-Desargues} has the geometric representation shown in Figure~\ref{fig:nondesargues}. It is a rank-3 matroid with 10 elements that is sparse paving with the 9 circuit-hyperplanes that are the 9 lines in the picture. Its Tutte polynomial is
\begin{equation}
T_M(x,y)= x^3+7x^2+19x+9xy+19y+21y^2+15y^3+10y^4+6y^5+3y^6+y^7.
\end{equation}

%\begin{eqnarray*}
%T_M(x,y)& = & (x-1)^3+10(x-1)^2+36x-135+9xy+201y+252(y-1)^2+\\
%        &   & 210(y-1)^3+120(y-1)^4+45(y-1)^5+10(y-1)^6+(y-1)^7\\
%        & = &x^3+7x^2+19x+9xy+19y+21y^2+15y^3+10y^4+6y^5+3y^6+y^7.
%\end{eqnarray*}

\begin{figure}[h!]
 \begin{center}
   \includegraphics[scale=0.4]{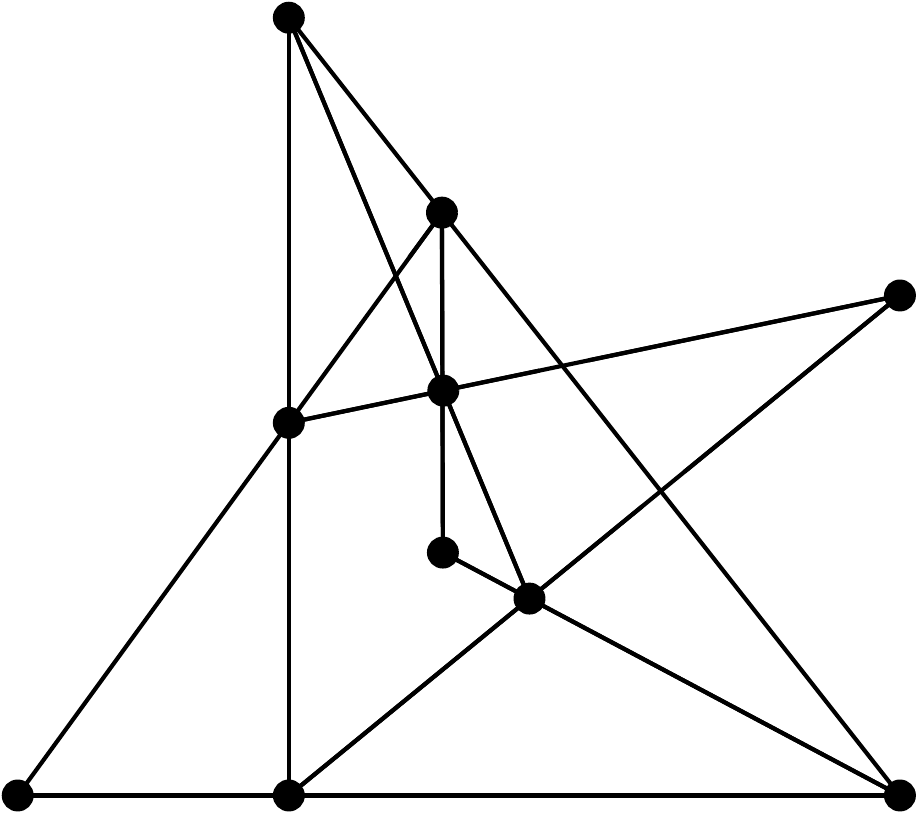}
    \caption{Geometric representation of the  non-Desargues matroid}
   \label{fig:nondesargues}
   \end{center}
\end{figure}

% steiner systems
\subsubsection*{The Steiner systems $S(2,3,13)$ and $S(5,6,12)$}
 Finally, we consider two Steiner systems from the Appendix of Oxley's book. In general any Steiner system $S(v-1,v,n)$ 
 is sparse paving, see~\cite{Wel10}. Thus, the Steiner triple system $S(2,3,13)$ is a rank-3 sparse paving matroid with
 13 elements and $\lambda=26$ circuit-hyperplanes, then its Tutte polynomial is 
\begin{eqnarray}
 T_{S(2,3,13)}(x,y)& = & x^3+10x^2+29x+26xy+29y+45y^2+36y^3\nonumber \\
                  &   & +28y^4+21y^5+15y^6+10y^7+6y^8+3y^9+y^{10}.
\end{eqnarray}

 The Steiner system $S(5,6,12)$  has rank 6, 12 elements and 132 circuit-hyperplanes. The Tutte polynomial is, by using equation~(\ref{eq:sparse_paving}),
\begin{eqnarray}
T_{S(5,6,12)}(x,y) 
           & = & x^6+6x^5+21x^4+56x^3+126x^2+120x+132xy\nonumber \\
           &  & +120y+126y^2+56y^3+21y^4+6y^5+y^6.
\end{eqnarray}

%%%%%%%%%%%%% Bibliography
%%%%%%%%%%%%%%%%%%%%%%%%%%%%%%%%%%%%

\end{document}